\documentclass[12pt, tikz,border=2mm]{article}

\usepackage{amsmath,amssymb,amsthm}
\usepackage{bbm}
\usepackage{url}
\usepackage[arrow,matrix,curve]{xy}
\usepackage{tikz}
\usepackage{tikz-cd}
\usepackage{hyperref}
\usepackage{listings}

\newtheorem{thm}{Theorem}[subsection]
\newtheorem{lem}[thm]{Lemma}
\newtheorem{introthm}{Theorem}

\newtheorem{prop}[thm]{Proposition}
\newtheorem{cor}[thm]{Corollary}

\newtheorem*{thm*}{Theorem}
\newtheorem*{prop*}{Proposition}
\newtheorem*{cor*}{Corollary}
\newtheorem*{conj*}{Conjecture}
\theoremstyle{definition}

\theoremstyle{definition}
\newtheorem{rmk}[thm]{Remark}

\theoremstyle{definition}
\newtheorem{df}[thm]{Definition}
\newcommand{\ZZ}{\mathbb Z}
\newcommand{\NN}{\mathbb N}
\newcommand{\RR}{\mathbb R}
\newcommand{\QQ}{\mathbb Q}
\newcommand{\CC}{\mathbb C}

\newcommand{\KK}{\mathbb K}

\def\sign{\mathop{\mathrm{sign}}\nolimits}

\def\diag{\mathop{\mathrm{diag}}\nolimits}

\renewcommand{\Re}{{\rm Re}}
\renewcommand{\Im}{{\rm Im}}

\tikzset{invclip/.style={clip,insert path={{[reset cm]
      (-\maxdimen,-\maxdimen) rectangle (\maxdimen,\maxdimen)
    }}}}

\begin{document}

\author{Hanno von Bodecker\footnote{Fakult{\"a}t f{\"u}r Mathematik, Universit{\"a}t Bielefeld, Germany} \and Sebastian Thyssen\footnote{Fakult\"at f\"ur Mathematik, Ruhr-Universit\"at Bochum, Germany}}

\title{Topological Automorphic Forms via Curves}

\date{}

\maketitle

\begin{abstract} 
We produce first examples of $p$-local height three TAF homology theories. The corresponding one-dimensional formal groups arise as  split summands of the formal groups of certain abelian three-folds, the Shimura variety of which can be reinterpreted as moduli of a family of Picard curves. This allows an explicit description of an automorphic form valued genus in terms of the coefficients of these curves. Moreover, our construction is such that the theories naturally come with restriction maps to TAF theories of lower height.
\end{abstract}


\section{Introduction}

There is a surprisingly strong connection between algebraic topology, more precisely stable homotopy theory, and algebraic geometry. At the heart of this connection lies the theory of formal groups. Being group objects in the category of formal schemes, formal groups are natural objects of interest in algebraic geometry. The biggest, or perhaps most important, class of examples arises via the completion-at-the-identity functor from the category of  abelian schemes.

In algebraic topology formal groups appear when evaluating complex orientable cohomology theories on the $H$-space $\mathbb{C}P^{\infty}$. An actual choice of complex orientation then corresponds to a choice of coordinates on the formal group and allows to express the group structure via a formal group law. There is a ring, the so called Lazard ring, which carries a universal formal group law and thus classifies formal group laws via ring homomorphisms out of it. On the side of topology it turns out \cite{Quillen:1969jh} that this ring is  the coefficient ring of the universal complex oriented cohomology theory $MU$ of complex cobordism. In fact, the stack associated to the Hopf algebroid comprising the Lazard ring and the coalgebra of strict isomorphisms of formal group laws 
 \textit{is} the moduli stack of formal groups.

It is worth noting that this moduli stack is stratified by an invariant called height and that the shadow of this stratification can be seen when dealing with Adams-Novikov spectral sequences which are based on complex oriented theories. In this context the filtration is called the chromatic height filtration.

There are many examples of complex oriented cohomology theories of the various heights, e.g.\ for every height there is a Morava $E$-theory (at every prime), but considering the connection between topology and algebraic geometry outlined above it is quite natural to ask, which of these complex oriented theories can be related not only to formal groups, but also to abelian schemes. 

From the algebraic geometric perspective the obvious starting point for any attempt to answer this question are elliptic curves, the $1$-dimensional abelian schemes. In \cite{Landweber:1995sw} the authors used bordism theory to construct complex oriented cohomology theories out of elliptic curves. More precisely, the evaluation of these \textit{elliptic} theories on $\mathbb{C}P^{\infty}$ turn out to be formal groups of elliptic curves over the coefficient ring of the theory. The development of the theory of elliptic cohomology culminated in the construction of the spectrum of topological modular forms. Despite being neither complex orientable nor elliptic, this spectrum can be interpreted as ``universal elliptic theory of height two'',  as it detects height two phenomena and has a (unique) map to every elliptic theory.

Formal groups of elliptic curves are of height one or two depending on whether the elliptic curve is ordinary of supersingular. From a moduli perspective the topologically interesting points are the supersingular points in $\mathcal{M}_{ell}$, the moduli stack of elliptic curves, and the ordinary locus is merely only necessary to interpolate between those and allow a flat map from  $\mathcal{M}_{ell}$ to  $\mathcal{M}_{FG}$, the moduli stack of formal groups.

As the height of the formal group of an abelian scheme is up to a factor of two bound by its dimension, one obviously has to increase the dimension of the abelian schemes  to pass beyond height two. Unfortunately, one then runs into difficulties on the topological side, because formal groups in topology turn out to be always one dimensional. In \cite{Behrens:2010aa} the authors considered Shimura varieties classifying only those $n$-dimensional abelian schemes which had enough structure to split off a one dimensional formal summand of the associated $n$-dimensional formal groups that had height up to $n$. Similar to the elliptic case the global sections are again not complex oriented themselves, but all \'etale opens of the derived Shimura variety give rise to complex oriented cohomology theories.

As beautiful as the theory is, as hard it is to provide examples. There have been a few investigations \cite{Hill:2010aa}, \cite{Behrens:2011aa}, \cite{Lawson:2015aa} of Shimura curves classifying abelian schemes up to dimension two, and thus cohomology theories of maximal height two. Also, it is hard to get any feeling for the formal group involved.
\bigskip

In this paper we produce a first example of a height three TAF theory by refining the approach taken in \cite{Bodecker:2016ad}. Over the Eisenstein integers  we access the Shimura varieties of structured abelian two and three-folds via the Torelli image of special families of curves of genus two and three, respectively. Though, the focus of the paper is the genus three case, we start with the two dimensional test case. Here, we present a thorough investigation of the geometry of the complex points of the Shimura variety and the automorphic forms, which is analogous to that of \cite{Bodecker:2016ad} in the Gaussian case. We introduce a genus $\varphi^L$ and use it to build the desired examples of topological automorphic form theories for this unitary group.
\begin{introthm}
Let $G$ be the unitary group $U(1,1;\mathbb{Z}[\zeta_3])$ and $M_{*}^{R}(G)$ be the ring of $R$-valued automorphic forms for $G$. 
\begin{itemize}
\item[i)] For all primes $p\equiv 1 \mod 3$, we have
\[
\varphi^L(BP_*) \subset M^{\mathbb{Z}_{(p)}}_*(G) \cong \mathbb{Z}_{(p)}[\kappa^2, \lambda] \subset \mathbb{Z}_{(p)}[\kappa, \lambda].
\]
\item[ii)] For $p=7,13$, the genus $\varphi^L$ gives $M_{*}^{\ZZ_{(p)}}(G)[\Delta_{6}^{-1}]$ the structure of a Landweber exact $BP_*$-algebra of height two. In particular, the functors
\[
TAF^{U(1,1;\ZZ[\zeta_3])}_{(p),*}(\cdot):=BP_{*}(\cdot)\otimes_{\varphi^L}M^{\ZZ_{(p)}}_{*}(G)[\Delta_{6}^{-1}]
\]
define a homology theory.
\end{itemize}
\end{introthm}
\bigskip

Before we turn to the three dimensional case of interest, note that there hasn't been any example of a height three TAF in the literature, so far. Let us also remark, that in this dimension things get less tractable, even over the complex numbers, in the sense that we can't actually fall back on modular forms to compute rings of automorphic forms ourselves. Instead, we rely on the analysis \cite{Holzapfel:1986aa,Shiga:1988aa} of the moduli of certain families of curves already investigated by Picard \cite{Picard:1883aa}.

In the case of inverted discriminant, i.e.\ smooth curves, the associated abelian schemes are irreducible. For primes $7$ and $13$ we produce a $v_3$-periodic homology theory by describing a genus $\varphi^{P}$ valued in automorphic forms for the unitary groups in question. 
\begin{introthm}\label{introthm B}
Let $\Gamma$ be the unitary group $U(2,1;\mathbb{Z}[\zeta_3])$ and $M_{*}^{R}(\Gamma)$ be the ring of $R$-valued automorphic forms for $\Gamma$. 
\begin{itemize}
\item[i)] For all primes $p\equiv 1 \mod 3$, we have
\[
\varphi^P(BP_*) \subset M^{\mathbb{Z}_{(p)}}_*(\Gamma) \cong \mathbb{Z}_{(p)}[G_2, G_4, G_3^2] \subset \mathbb{Z}_{(p)}[G_2, G_3, G_4] .
\]
\item[ii)] For $p=7,13$, the genus $\varphi^P$ gives $M_{*}^{\ZZ_{(p)}}(\Gamma)[\Delta_{C}^{-1}]$ the structure of a Landweber exact $BP_*$-algebra of height three. In particular, the functors
\[
TAF^{U(2,1;\ZZ[\zeta_3])}_{(p),*}(\cdot):=BP_{*}(\cdot)\otimes_{\varphi^P}M^{\ZZ_{(p)}}_{*}(\Gamma)[\Delta_{C}^{-1}]
\]
define a homology theory.
\end{itemize}
\end{introthm}

However our genus behaves even better. In fact it makes sense on more than the locus of smooth curves, which means that one is fine with inverting less than the (vanishing locus of the) discriminant. Analyzing congruences between automorphic forms we find that inverting the form $G_{3}$ (which is of lower weight than the discriminant) is sufficient for the resulting theory to be $v_3$-periodic. Since we invert less than before, this clearly imports certain types of singular curves into the picture. We investigate the different types of singularities for Picard curves and how they relate this genus and its height three theories to genera and theories for smaller unitary groups and lower chromatic height.
\bigskip

To emphasize this more, let us describe the situation from a different point of view. The configuration space of five (ordered) distinct points in $\mathbb{P}^1$ corresponds to the space of smooth Picard curves of genus $3$ (with markings). This space can thus be identified with the open dense subspace of irreducible varieties in the Shimura variety $Sh$ by considering the Jacobians of these Picard curves. Moreover, allowing points in the configuration space to come together creates singular Picard curves, which, though not possessing Jacobians themselves, can be associated to points in the complement of the irreducible locus of $Sh$. As we keep infinity as distinct point and also do not allow all the remaining four points to coincide, there are precisely three possible cases of how points can come together. These different degenerations types (of curves) yield a stratification of the (compactified) Shimura variety. The case of two points coming together is described by a $U(1,1)$-stratum in $Sh$. The inclusion of this stratum induces restriction morphisms of automorphic forms. In particular, the standard inclusion 
$$ U(1,1; \mathbb{Z}[\zeta_3]) \to U(2,1; \mathbb{Z}[\zeta_3])$$
yields a map
$$ r \colon M^{\ZZ_{(p)}}_{*}(U(2,1; \mathbb{Z}[\zeta_3]))\to M^{\ZZ_{(p)}}_{*}(U(1,1; \mathbb{Z}[\zeta_3])). $$

Our genera and the associated homology theories are well behaved with respect to this restriction morphism.

\begin{introthm}
Let $r$ be the restriction map defined above.
\begin{itemize}
\item[i)] For all primes $p\equiv 1 \mod 3$  the genus $r\circ \varphi^P$ is equivalent to $\varphi^L$; more precisely, the associated group laws are isomorphic over $\ZZ_{(p)}[\kappa,\lambda]$. 
\item[ii)] For $p=7, 13$ the homomorphism $r$ induces a morphism of Landweber exact algebras $$ \mathbb{Z}_{(p)}[G_2, G_3^{\pm 1}, G_4] \to \mathbb{Z}_{(p)}[\kappa^{\pm 1}, \lambda^{\pm 1}]$$ decreasing height by one.
\end{itemize}
\end{introthm}

\bigskip

We also investigate a level structure closely related to the special level structure $U(2,1;\ZZ[\zeta_3])[\sqrt{-3}]$ and find the analog of  Theorem \ref{introthm B} to hold in this context; we refer the reader to the appendix for the precise statement.

Note  that on the underlying algebraic geometric side of things, much of the necessary mathematics, e.g. the geometry of the moduli of these curves, is classical \cite{Picard:1883aa}, \cite{Bolza:1887aa} or at least known for quite a while \cite{Shiga:1988aa}. A thorough and modern account can be found in \cite{Holzapfel:1986aa} and \cite{Holzapfel:1995aa}.

\subsection*{Acknowledgements}

We learned that Tyler Lawson independently proved a $p$--complete version of Theorem \ref{introthm B}, although by slightly different means, and would like to thank him for helpful feedback on an earlier version of this paper.

The second author thanks the DFG SPP 1786 for financial support and the Ruhr--Universit\"at Bochum for hospitality.


%
%
%
%
%
%
%
%
%
%
%
%
%
%
%
%
%
%
%
\section{Some Shimura varieties of unitary type}		\label{SectionUnitarySymplecticGroups}
%
%
%
%
%
%
%
%
\subsection{Symplectic groups and abelian varieties}		\label{SectionSymplecticGroups}
%
%
%
%
The matrix $\left( \begin{smallmatrix} &  \mathbbm{1}\\ -\mathbbm{1} & \end{smallmatrix} \right)$, where $\mathbbm{1}$ denotes the unit matrix of rank $n$,
defines a skew bilinear form on $\RR^{2n}$. The symplectic group is defined as
\[
Sp(n)=\{g\in GL(2n;\RR):g^{tr}\left( \begin{smallmatrix} &  \mathbbm{1}\\ -\mathbbm{1} & \end{smallmatrix} \right) g=\left( \begin{smallmatrix} &  \mathbbm{1}\\ -\mathbbm{1} & \end{smallmatrix} \right)\};
\]
it is a non-compact real Lie group of dimension $2n^{2}+n$ and real rank $n$. Moreover, we have the inversion formula
\[
\left(\begin{smallmatrix}A&B\\C&D\end{smallmatrix}\right)^{-1}=\left(\begin{smallmatrix}D^{tr}&-B^{tr}\\-C^{tr}&A^{tr}\end{smallmatrix}\right).
\]
Since $\left( \begin{smallmatrix} &  \mathbbm{1}\\ -\mathbbm{1} & \end{smallmatrix} \right)^{-1}=-\left( \begin{smallmatrix} &  \mathbbm{1}\\ -\mathbbm{1} & \end{smallmatrix} \right)$, the group $Sp(n)$ is closed under transposition. Thus, $g\mapsto (g^{tr})^{-1}$ is an involutive automorphism. Let $$Sp(n)\cap O(2n)\cong U(n)$$ be the fixed point set; it is a maximal compact subgroup, and there is a well-known identification between the associated symmetric space and the Siegel space $\mathbb{S}_{n}$ consisting of symmetric complex $n$-by-$n$ matrices with positive-definite imaginary part,
\[
Sp(n)/U(n)\cong \mathbb{S}_{n}=\{\Omega=\Re(\Omega)+i\Im(\Omega):\Omega^{tr}=\Omega,\ \Im(\Omega)>0\}.
\]
The Siegel upper half space $\mathbb{S}_n$ is a hermitian symmetric domain well known to parametrize (complex) abelian varieties equipped with a principal polarization: For $\Omega\in\mathbb{S}_{n}$, the columns of the matrix $( \mathbbm{1},\Omega)$ constitute a $\ZZ$--module basis for a symplectic lattice 
\[
\Lambda = \mathbb{Z}^n \oplus \Omega\mathbb{Z}^n
\]
in $\mathbb{C}^n$, and the quotient $\CC^{n}/\Lambda$ becomes a principally polarized abelian variety.  Letting 
\[
\sigma\colon Sp(n)\to Sp(n), \quad \sigma(\begin{smallmatrix}A&B\\C&D\end{smallmatrix})=(\begin{smallmatrix}A&-B\\-C&D\end{smallmatrix}),
\]
denote the automorphism introducing a sign in the off-diagonal blocks, we have the left action
\[
( \mathbbm{1},\Omega)\sigma(\begin{smallmatrix}A&B\\C&D\end{smallmatrix})^{-1}=(D^{tr}+\Omega C^{tr}, B^{tr}+\Omega A^{tr}),
\]
allowing us to recover the usual fractional linear transformation on $\mathbb{S}_{n}$,
\[
\left(\begin{smallmatrix}A&B\\ C&D\end{smallmatrix}\right)\cdot\Omega=(A\Omega+B)(C\Omega+D)^{-1}.
\]
Clearly, $Sp(n;\ZZ)$--equivalent points in Siegel space determine isomorphic abelian varieties.

\begin{rmk}\label{relative tangent bundle}
Observe that
\begin{align*}
((C\Omega+D)^{tr})^{-1}(\mathbbm{1},\Omega)&=((C\Omega+D)^{tr})^{-1}(\mathbbm{1},\Omega)\left(\begin{smallmatrix}A&-B\\-C&D\end{smallmatrix}\right)^{-1}\left(\begin{smallmatrix}A&-B\\-C&D\end{smallmatrix}\right)\\
&=(\mathbbm{1},(A\Omega+B)(C\Omega+D)^{-1})\left(\begin{smallmatrix}A&-B\\-C&D\end{smallmatrix}\right)
\end{align*}
This has the following interpretation: The element $g=\left(\begin{smallmatrix}A&B\\C&D\end{smallmatrix}\right)\in Sp(n)$ induces a map from the ppav determined by $\Omega$ to the ppav determined by $\Omega'=(A\Omega+B)(C\Omega+D)^{-1}$, and  $((C\Omega+D)^{tr})^{-1}$ can be identified with the induced map on tangent spaces at the origin. In particular, we have a left action of $Sp(n;\ZZ)$ on $\mathbb{S}_{n}\times\CC^{n}$ given by
\begin{equation}\label{almost a bundle of Lie algebras}
\left(\begin{smallmatrix}A&B\\ C&D\end{smallmatrix}\right)\cdot(\Omega,\vec{w})=\big((A\Omega+B)(C\Omega+D)^{-1},((C\Omega+D)^{tr})^{-1}\vec{w} \big),
\end{equation}
and the quotient tries to be the bundle of Lie algebras of the respective abelian varieties.
\end{rmk}

%
%
%
%
%
%
%
%
%
\subsection{Hermitian lattices and unitary groups}		\label{SectionUnitaryGroups}
 In this section we consider the case of  imaginary quadratic number fields $\KK=\QQ(\sqrt{-N})$ with class number one and odd discriminant (the case of even discriminant is analogous, for more details we refer the reader to a forthcoming paper). 
In the odd discriminant case, $1$ and $\alpha=-\frac{1}{2}+\frac{\sqrt{-N}}{2}$ constitute a $\ZZ$--module basis for $\mathcal{O}$, the ring of integers of $\KK$. Since we assume that the class number is one, a hermitian $\mathcal{O}$-lattice is a free $\mathcal{O}$-module  equipped with a hermitian form. It is called \textit{integral} if the form takes values in $\mathcal{O}$, \textit{unimodular} if it is self-dual. Equivalently, an integral lattice is unimodular if the determinant of a Gram matrix is invertible. 

Over $\KK$, the hermitian form is diagonalizable; thus, it has a well-defined signature which we write as a pair of non-negative numbers. Recall that for any {\em indefinite} signature, there is a unique hermitian unimodular $\mathcal{O}$--lattice (this can be proved as in the case of unimodular quadratic lattices over $\ZZ$, noting that  for the odd-discriminant case, any unimodular lattice is odd, i.e., contains a vector of odd norm).

\bigskip

Let $H\in GL(n;\ZZ)$ be symmetric and self-inverse, $H=H^{tr}=H^{-1}$. Thus, $H$ can be considered as the Gram matrix of a unimodular hermitian $\mathcal{O}$-lattice. Define the corresponding unitary group of the hermitian form $H$ by $$U(H)=\{g\in GL(n;\CC):g^{\dagger}Hg=H\}.$$ It is a reductive real Lie group. We assume that $H$ has indefinite signature $(r,s)$, so that $U(H)\cong U(r,s)$ is non-compact. 

Due to our assumption on $H$, we have the inversion formula $g^{-1}=Hg^{\dagger}H$; moreover, the group $U(H)$ is closed under conjugation and transposition. In particular, we have an involutive automorphism $g\mapsto(g^{\dagger})^{-1}=HgH$, and its fixed points constitute a maximal compact subgroup (isomorphic to $U(r)\times U(s)$). 

Clearly, the group $$U(H;\mathcal{O})=U(H)\cap GL(n;\mathcal{O})$$ can be identified with the isometry group of the unimodular hermitian lattice determined by $H$.

%
%
%
%
%
%
%
%
%
\subsection{Standard embedding}			\label{SectionEmbeddings}
%
%
%
%
%

We return to the group $U(H;\mathcal{O})$, where $H\in GL(n;\ZZ)$ is symmetric and self-inverse. Considering $\mathcal{O}$ as $\ZZ$-module with basis $\{1, \alpha=-\frac{1}{2}+\frac{\sqrt{-N}}{2}\}$, the imaginary part of the hermitian form defines a skew form on the lattice. Taking the $\ZZ$--basis $e_{1},\dots, e_{n},He_{1}\alpha,\dots,He_{n}\alpha$, the skew $\ZZ$--bilinear form $$(U,V)\mapsto (U^{\dagger}HV-V^{\dagger}HU)/\sqrt{-N}$$ takes the standard form. Thus, we have an embedding $$U(H;\mathcal{O})\to Sp(n;\ZZ)$$ given by
\[
X+\alpha Y\mapsto\begin{pmatrix}X&-\frac{N+1}{4}YH\\HY&H(X-Y)H\end{pmatrix};
\]
we compose this with the involutive automorphism 
\[
\sigma\colon Sp(n)\to Sp(n), \quad \sigma(\begin{smallmatrix}A&B\\C&D\end{smallmatrix})=(\begin{smallmatrix}A&-B\\-C&D\end{smallmatrix}),
\]
to define the injective homomorphism
\[
\rho\colon U(H)\to Sp(n),\quad X+\alpha Y\mapsto\begin{pmatrix}X&\frac{N+1}{4}YH\\-HY&H(X-Y)H\end{pmatrix}.
\]
To characterize the image, let $J_{N}=\left(\begin{smallmatrix}&-{\textstyle{\frac{N+1}{4}}}H\\H&-\mathbbm{1}\end{smallmatrix}\right)$ be the matrix representing multiplication by $\alpha$ with respect to the above $\ZZ$--basis. Then we have:
\begin{prop}
\[
\rho\colon U(H)\;\xrightarrow{\cong}\;\left\{g\in Sp(n):\sigma(g)J_{N}=J_{N}\sigma(g)\right\}
\]
\end{prop}
\begin{proof} Consider $(\begin{smallmatrix}A&B\\C&D\end{smallmatrix})\in Sp(n)$ and compute
\begin{align*}
\begin{pmatrix}A&-B\\-C&D\end{pmatrix}\begin{pmatrix}&-\textstyle{\frac{N+1}{4}}H\\H&-\mathbbm{1}\end{pmatrix}&=\begin{pmatrix}-BH&-\textstyle{\frac{N+1}{4}}AH+B\\DH&\textstyle{\frac{N+1}{4}}CH-D\end{pmatrix}
\\
\begin{pmatrix}&-\textstyle{\frac{N+1}{4}}H\\H&-\mathbbm{1}\end{pmatrix}\begin{pmatrix}A&-B\\-C&D\end{pmatrix}&=\begin{pmatrix}\textstyle{\frac{N+1}{4}}HC&-\textstyle{\frac{N+1}{4}}HD\\HA+C&-HB-D\end{pmatrix}
\end{align*}
Thus, $B=-\textstyle{\frac{N+1}{4}}HCH$, $D=HAH+CH$. Now use the inversion formula
\[
\begin{pmatrix}A&-\textstyle{\frac{N+1}{4}}HCH\\C&HAH+CH\end{pmatrix}^{-1}=\begin{pmatrix}HA^{tr}H+HC^{tr}&\textstyle{\frac{N+1}{4}}HC^{tr}H\\-C^{tr}&A^{tr}\end{pmatrix}
\]
and put $C=-HY$ to recover the defining equations of $U(H)$,
\[
H=X^{tr}HX+\textstyle{\frac{N+1}{4}}Y^{tr}HY-Y^{tr}HX+\alpha(X^{tr}HY-Y^{tr}HX).\qedhere
\]
\end{proof}
Next, observe that the stabilizer of
\begin{equation}\label{the base point}
\Omega_{0}=\frac{1}{2}(-H+\sqrt{-N}\mathbbm{1})\in\mathbb{S}_{n}
\end{equation}
in $\rho (U(H))$ is the image  of the subgroup fixed by the involutive automorphism
\[
U(H)\to U(H),\qquad g\mapsto HgH,
\] under $\rho$. To see this, write $\left(\begin{smallmatrix}A&B\\C&D\end{smallmatrix}\right)$ with $B=-\frac{N+1}{4}HCH$, $D=HAH+CH$, then the stability requirement becomes $C=HCH$ and $A=HAH$, i.e.\  $g=HgH$, for $g\in U(H)$.

\bigskip

Since
\[
(\mathbbm{1},\Omega_{0})J_{N}=(\Omega_{0}H,-{\textstyle{\frac{N+1}{4}}}H-\Omega_{0})
\]
and $H\in GL(n;\ZZ)$, the $\ZZ$--lattice in $\CC^{n}$ determined by the span of the columns of $(\mathbbm{1},\Omega_{0})$ is an $\mathcal{O}$--module. Moreover, $$\Omega_{0}H\Omega_{0}=-{\textstyle{\frac{N+1}{4}}}H-\Omega_{0}.$$ Thus, the endomorphism $\CC^{n}\to\CC^{n}$ given by left multiplication by $\Omega_{0}H$ descends to an endomorphism of the abelian variety determined by $\Omega_{0}\in \mathbb{S}_{n}$.

\bigskip

This remains true for any lattice in the $\rho (U(H))$--orbit of $(\mathbbm{1},\Omega_{0})$: we have
\[
(\mathbbm{1},\Omega_{0})\cdot J_{N}=\Omega_{0}H\cdot(1,\Omega_{0})
\]
and
\[
(\mathbbm{1},\Omega_{0})\begin{pmatrix}D^{tr}&B^{tr}\\C^{tr}&A^{tr}\end{pmatrix}=(C\Omega_{0}+D)^{tr}(\mathbbm{1},\tilde{\Omega}),
\]
where $\tilde{\Omega}=(A\Omega_{0}+B)(C\Omega_{0}+D)^{-1}$. Thus, for $\left(\begin{smallmatrix}A&B\\C&D\end{smallmatrix}\right)\in\rho U(H)$ we have
\[
(\mathbbm{1},\tilde{\Omega})\cdot J_{N}=(((C\Omega_{0}+D)^{tr})^{-1}\Omega_{0}H(C\Omega_{0}+D)^{tr})\cdot(\mathbbm{1},\tilde{\Omega}),
\]
and, using the defining equations for $U(H)$, we also have
\[
\Omega_{0}H(C\Omega_{0}+D)^{tr}=(B^{tr}+\Omega_{0}A^{tr})H,
\]
i.e.\ $((C\Omega_{0}+D)^{tr})^{-1}\Omega_{0}H(C\Omega_{0}+D)^{tr}=((C\Omega_{0}+D)^{tr})^{-1}(B^{tr}+\Omega_{0}A^{tr})H=\tilde{\Omega}H$.

\bigskip

Summarizing, the abelian varieties in the image of $U(H)/K\to \mathbb{S}_{r+s}$ have an $\mathcal{O}$-module structure, and the base point can be identified with \eqref{the base point}.

%
%
%
%
%
\subsection{Signature $(n-1,1)$}\label{SectionSignature n-1,1}
%
%
%
%
\subsubsection*{The complex hyperbolic spaces}	
%
%
%
%

For definiteness, let us fix the matrix
\[
H_{n}=\begin{pmatrix}&&1\\&\mathbbm{1}_{n-2}&\\1&&\end{pmatrix}
\]
of signature $(n-1,1)$. The group $U(H_{n})$ acts on the complex projective space $\CC P^{n-1}$ (from the left) and preserves the open set consisting of lines in $\CC^{n}$ on which $H_{n}$ is negative definite, i.e.\ a model of complex hyperbolic space,
\[
\CC H^{n-1}\cong\{[Z]\in\CC P^{n-1}: Z^{\dagger}H_{n}Z<0\}.
\] 
Restricting attention to representatives with last component equal to one, $$[Z]=[\left(\begin{smallmatrix}z_{0}\\z_{1}\\1\end{smallmatrix}\right)],$$ we obtain a description as a left half-space,
\[
\CC H^{n-1}\cong \mathbb{L}_{n-1}=\{\left(\begin{smallmatrix}z_{0}\\z_{1}\end{smallmatrix}\right)\in\CC\times\CC^{n-2}:z_{0}+z_{0}^{\dagger}+z_{1}^{\dagger}z_{1}<0\};
\]
in this realization, the base point (i.e.\  the choice of maximal compact subgroup $U(H_{n})\cap U(n)\subset U(H_{n})$) corresponds to $\left(\begin{smallmatrix}-1\\0\end{smallmatrix}\right)\in\mathbb{L}_{n-1}$.

\bigskip

The group $U(H_{n};\mathcal{O})$ acts on this space in a fractional linear fashion.

\begin{rmk} The space $\mathbb{L}_{n-1}/U(H_{n};\mathcal{O})$ is not compact; its cusps can be identified with the $U(H_{n};\mathcal{O})$--equivalence classes of $H_{n}$--isotropic lines in $\mathbb{K}^{n}$, which in turn can be identified with the isometry classes of positive-definite unimodular hermitian lattices of rank $n-2$.
\end{rmk}

\bigskip

Inserting a column/row with only entry $1$ in penultimate position leads to an embedding $$ U(H_{n-1}) \to U(H_n)$$ of unitary groups, exemplarily presented for $n=3$ by 
$$ \begin{pmatrix}a&b\\c&d\end{pmatrix} \mapsto \begin{pmatrix}a&&b\\&1&\\c&&d\end{pmatrix}.$$
This in turn induces an embedding 
\[
\mathbb{L}_{n-2} \to \mathbb{L}_{n-1},\quad \left(\begin{smallmatrix}z_{0}\\z_{1}\end{smallmatrix}\right) \mapsto \left(\begin{smallmatrix}z_{0}\\z_{1}\\0\end{smallmatrix}\right),
\]
of complex hyperbolic spaces.

%
%
%
%
\subsubsection*{Induced standard embedding}
%
%
%
In this section we consider the embedding on the level of spaces induced by the standard embedding on the level of groups.

\bigskip

For $\KK=\QQ(\sqrt{-N})$ as above, put $\alpha=(-1+\sqrt{-N})/2$ (so that $\alpha^{2}+\alpha+(N+1)/4=0$). We recall our standard embedding
\[
\rho\colon U(H_{n})\to Sp(n),\ X+\alpha Y\mapsto\begin{pmatrix}X&\textstyle{\frac{N+1}{4}}YH_{n}\\-H_{n}Y&H_{n}(X-Y)H_{n}\end{pmatrix}.
\]

\begin{prop} The standard embedding $\rho\colon U(H_{3})\to Sp(3)$ induces the embedding
\[
\rho_{*}\colon\mathbb{L}_{2}\to\mathbb{S}_{3},\quad\begin{pmatrix}z_{0}\\z_{1}\end{pmatrix}\mapsto\frac{1}{2z_{0}+z_{1}^{2}}\begin{pmatrix}-\sqrt{-N}z_{0}^{2}&-\sqrt{-N}z_{0}z_{1}&-z_{0}+\alpha z_{1}^{2}\\-\sqrt{-N}z_{1}z_{0}&2\alpha z_{0}+\bar\alpha z_{1}^{2}&-\sqrt{-N}z_{1}\\-z_{0}+\alpha z_{1}^{2}&-\sqrt{-N}z_{1}&-\sqrt{-N}\end{pmatrix}
\]
\end{prop}

\begin{proof}
We start proving the special case $\mathbb{L}_1 \to \mathbb{S}_3$. Herefore, note that $U(H_{2})$ acts transitively on $\mathbb{L}_{1}$. Writing $z_{0}=-a^{2}+\sqrt{-N}b$, for real $a$ and $b$, observe that
\[
\begin{pmatrix}1&\sqrt{-N}b\\&1\end{pmatrix}\begin{pmatrix}-a^{2}\\1\end{pmatrix}=\begin{pmatrix}-a^{2}+\sqrt{-N}b\\1\end{pmatrix},
\]
and, as $\sqrt{-N}=2\alpha+1$,
\[
\rho\colon\begin{pmatrix}a&\sqrt{-N}b/a\\0&1/a\end{pmatrix}\mapsto\begin{pmatrix}a&b/a&\textstyle{\frac{N+1}{4}}2b/a&0\\0&1/a&0&0\\0&0&1/a&0\\0&-2b/a&-b/a&a\end{pmatrix}
\]
Now, consider the fractional linear action on the basepoint
\[
\frac{1}{2}(-H_{2}+\sqrt{-N}\mathbbm{1})=\frac{1}{2}\begin{pmatrix}\sqrt{-N}&-1\\-1&\sqrt{-N}\end{pmatrix}\in\mathbb{S}_{2};
\]
this yields
\[
\frac{1}{2}\begin{pmatrix}a\sqrt{-N}+Nb/a&-a+\sqrt{-N}b/a\\-1/a&\sqrt{-N}/a\end{pmatrix}\begin{pmatrix}1/a&\\&a-\sqrt{-N}b/a\end{pmatrix}^{-1},
\]
which indeed equals
\[
\frac{1}{2}\begin{pmatrix}\sqrt{-N}(a^{2}-\sqrt{-N}b)&-1\\-1&\frac{\sqrt{-N}}{a^{2}-\sqrt{-N}b}\end{pmatrix}.
\]
For the general case, note that we have a Heisenberg translation acting via $$ \begin{pmatrix} 1&-z_1^{tr}&-\tfrac{1}{2} |z_1|^2\\ &1&z_1\\ &&1\end{pmatrix}\begin{pmatrix} z_0' \\ 0\\ 1 \end{pmatrix} =  \begin{pmatrix} z_0' - \tfrac{1}{2}|z_1|^2 \\ z_1\\1\end{pmatrix}.$$
Thus it suffices to put $z_0 = z_0' - \tfrac{1}{2}|z_1|^2$ and check the action of the image of the Heisenberg translation under $\rho$.
\end{proof}

Working with the diagonal form $H_{n}'=\diag(1,\dots,1,-1)$, the variety corresponding to the basepoint is the product of $n$ elliptic curves, each endowed with a multiplication by $\mathcal{O}$; up to isomorphism, this variety can also be obtained using the embedding based on the form $H_{n}$.

\bigskip

\begin{prop}
The abelian three-fold corresponding to $\rho_*(\left(\begin{smallmatrix} \bar{\alpha} \\ 0 \end{smallmatrix}\right))$ is isomorphic to a product of three elliptic curves with complex multiplication; more precisely, the action of $J_{N}$ induces multiplication by $\alpha$ on two of these curves, and multiplication by $\bar\alpha$ on the third.
\end{prop}
\begin{proof}
In view of the embedding $\mathbb{L}_{1}\to\mathbb{L}_{2}$, it suffices to prove the analogous result for signature $(1,1)$: In this case, we have
\[
\begin{pmatrix}-\bar\alpha&1\\\alpha&1\end{pmatrix}\begin{pmatrix}&1\\1&\end{pmatrix}\begin{pmatrix}-\alpha&\bar\alpha\\1&1\end{pmatrix}=\begin{pmatrix}1&\\&-1\end{pmatrix}.
\]
Therefore we have a map
\[
\beta\colon U(H_{2})\to U(1,1),\quad  g \mapsto \begin{pmatrix}1&-\bar\alpha\\-1&-\alpha\end{pmatrix}g\begin{pmatrix}-\alpha&\bar\alpha\\1&1\end{pmatrix}=\hat g,
\]
and an induced map $$\mathbb{L}_{1}\to \mathbb{B}_{1}=\{z'\in\CC:|z'|^{2}<1\}.$$ Under the inverse mapping, the base point of the ball gets mapped to
\[
z'=0\mapsto z_{0}=\bar\alpha.
\]
The unimodular $\mathcal{O}$--lattices defined by $H_{2}$ and $H_{2}'=\diag(1,-1)$ are isometric; therefore, their respective standard embeddings are related by an integrally symplectic base change. More precisely, looking at the diagram
\[
\xymatrix{U(H_{2})\ar[r]^{\beta}\ar[d]^{\rho_{H_{2}^{}}}&U(H_{2}')\ar[d]^{\rho_{H_{2}'}}\\Sp(2)\ar[r]^{\hat\beta}&Sp(2)}
\]
we see that $\hat\beta\colon \rho_{H_{2}^{}}(U(H_{2}))\to\rho_{H_{2}'}(U(H_{2}'))$ is represented by {\em right} conjugation by the symplectic matrix
\[
\tilde{g}=\left(\begin{smallmatrix}1&0& &\\0&1&&\\&&0&-1\\&&-1&0\end{smallmatrix}\right)\left(\begin{smallmatrix}0&-1&\frac{N+1}{4}&-\frac{N+1}{4}\\1&1&0&0\\-1&-1&1&0\\0&0&1&1\end{smallmatrix}\right)\left(\begin{smallmatrix}1&0&&\\0&1&&\\&&-1&0\\&&0&1\end{smallmatrix}\right)=\left(\begin{smallmatrix}0&-1&-\frac{N+1}{4}&\frac{N+1}{4}\\1&1&0&0\\0&0&1&-1\\1&1&1&0\end{smallmatrix}\right),
\]
and we have $\sigma(\tilde{g})^{-1}J_{N}\sigma({\tilde{g}})=\left(\begin{smallmatrix}0&-\tfrac{N+1}{4}H_{2}'\\H_{2}'&-\mathbbm{1}\end{smallmatrix}\right)=J_{N}'$, which proves the claim.
\end{proof}

\subsubsection*{A twisted embedding}		\label{SectionTwistedEmbedding}
We define a twisted embedding $\iota\colon U(H_n) \to Sp(n)$ on the level of groups, exemplarily the case of $n=3$, by
\[
\iota(g)=t\rho(g)t^{-1}
\]
where
\[
t=\begin{pmatrix}1&&&0&&\\&1&&&0&\\&&0&&&-1\\0&&&1&&\\ &0&&&1& \\&&1&&&0\end{pmatrix}\in Sp(3;\ZZ);
\]
is the twisting matrix. Direct computation shows:

\begin{prop} The induced twisted embedding $\mathbb{L}_{2}\to\mathbb{S}_{3}$ is given by
\[
\iota_{*}\colon\begin{pmatrix}z_{0}\\z_{1}\end{pmatrix}\mapsto\begin{pmatrix}\frac{\frac{N+1}{2}z_{0}+\alpha^{2}z_{1}^{2}}{\sqrt{-N}}&-\alpha z_{1}&\frac{z_{0}-\alpha z_{1}^{2}}{\sqrt{-N}}\\-\alpha z_{1}&\alpha&z_{1}\\\frac{z_{0}-\alpha z_{1}^{2}}{\sqrt{-N}}&z_{1}&\frac{2z_{0}+z_{1}^{2}}{\sqrt{-N}}\end{pmatrix}.
\]\qed
\end{prop}

The twisted endomorphism becomes
\[
\sigma(t)J_{N}\sigma(t^{-1})=\begin{pmatrix}0&&-\frac{N+1}{4}&0&&0\\&0&&&-\frac{N+1}{4}&\\1&&-1&0&&0\\0&&0&-1&&-1\\&1&&&-1&\\0&&0&\frac{N+1}{4}&&0\end{pmatrix};
\]
consequently, multiplication by $\alpha$ is represented by
\[
\mu_{\alpha}=\begin{pmatrix}0&-\alpha z_{1}&-\frac{N+1}{4}\\0&\alpha&0\\1&z_{1}&-1\end{pmatrix},
\]
which indeed satisfies $\mu_{\alpha}+\mu_{\alpha}^{2}=-\frac{N+1}{4}$, and $\mu_{\alpha}\left(\begin{smallmatrix}-\alpha\\0\\1\end{smallmatrix}\right)=\bar\alpha\left(\begin{smallmatrix}-\alpha\\0\\1\end{smallmatrix}\right)$. Furthermore, we clearly have $\mu_{\alpha}\left(\begin{smallmatrix}1+\alpha\\0\\1\end{smallmatrix}\right)=\alpha\left(\begin{smallmatrix}1+\alpha\\0\\1\end{smallmatrix}\right)$ and $\mu_{\alpha}\left(\begin{smallmatrix}0\\1\\0\end{smallmatrix}\right)=\alpha\left(\begin{smallmatrix}0\\1\\0\end{smallmatrix}\right)$.
\bigskip

There is a left action of the unitary group $U(H_3)$ on the space $\mathbb{L}_2\times \mathbb{C}$, which for $g = \left( g_{kl} \right)_{kl} \in U(H_3)$ is given by
$$ g.((\begin{smallmatrix} z_0\\z_1\end{smallmatrix}), w) = (g. (\begin{smallmatrix}z_0\\z_1\end{smallmatrix}), (g_{31}z_0 + g_{32}z_1 + g_{33})^{-1}w).$$ There is also an induced action of $U(H_3)$ on the space $\mathbb{S}_3\times \mathbb{C}^3$ via the map $\iota$, which is the restriction of the action \eqref{almost a bundle of Lie algebras}.
\begin{prop} The map
\begin{align*}
\mathbb{L}_2 \times \mathbb{C} &\to \mathbb{S}_3 \times \mathbb{C}^3, \\
((\begin{smallmatrix} z_0 \\ z_1\end{smallmatrix}), w) &\mapsto \left( \iota_{*} \left(\begin{smallmatrix}z_{0}\\z_{1}\end{smallmatrix}\right), \left( \begin{smallmatrix} -\alpha \\ 0 \\1\end{smallmatrix}\right)\cdot w\right)
\end{align*}
is $U(H_3)$-equivariant. Thus, in particular, the line spanned by $E_{\bar\alpha}=\left(\begin{smallmatrix}-\alpha\\0\\1\end{smallmatrix}\right)$ is invariant. 
\end{prop}
\begin{proof}
Let us sketch the direct computation of the claim. We have to show 
\begin{align}	\label{Equivariance}
\left((C\Omega + D)^{tr} \right)^{-1}\cdot E_{\bar{\alpha}} &= (g_{31}z_0 + g_{32}z_1 + g_{33})^{-1} \cdot E_{\bar{\alpha}} \nonumber \\
\Leftrightarrow \qquad (C\Omega + D)^{tr} \cdot E_{\bar{\alpha}} &= (g_{31}z_0 + g_{32}z_1 + g_{33}) \cdot E_{\bar{\alpha}}.
\end{align}
Having 
\[
g=X+\alpha Y\mapsto t\left(\begin{smallmatrix} X&\frac{N+1}{4}YH_{3}\\-H_{3}Y&H_{3}(X-Y)H_{3}\end{smallmatrix}\right)t^{-1}=\left(\begin{smallmatrix}A&B\\C&D\end{smallmatrix}\right),
\]
it is convenient to put $\hat g=\left(\begin{smallmatrix}X&\frac{N+1}{4}Y\\-Y&X-Y\end{smallmatrix}\right)$ in order to identify the transpose of the left hand side of \eqref{Equivariance} with 
\begin{align*}
\begin{pmatrix}0&E_{\bar\alpha}^{tr}\end{pmatrix}\begin{pmatrix}A&B\\C&D\end{pmatrix}\begin{pmatrix}\Omega\\\mathbbm{1}\end{pmatrix} &=
\begin{pmatrix}0&E_{\bar\alpha}^{tr}\end{pmatrix} t \begin{pmatrix}\mathbbm{1}&\\&H_{3}\end{pmatrix} \hat{g}  \begin{pmatrix}\mathbbm{1}&\\&H_{3}\end{pmatrix}  t^{-1} \begin{pmatrix}\Omega\\\mathbbm{1}\end{pmatrix}.
\end{align*}
Now compute
\[
\begin{pmatrix}0&E_{\bar\alpha}^{tr}\end{pmatrix} t \begin{pmatrix}\mathbbm{1}&\\&H_{3}\end{pmatrix} \hat{g}  \begin{pmatrix}\mathbbm{1}&\\&H_{3}\end{pmatrix}  t^{-1}= (g_{31},\: g_{32},\: \alpha g_{31},\: -\alpha g_{33} ,\: -\alpha g_{32} ,\: g_{33}),
\]
and multiply by $\left(\begin{smallmatrix}\Omega\\\mathbbm{1}\end{smallmatrix}\right)$ to find the claim. It is useful to remember the equality $-\alpha^2 = \alpha + \tfrac{N+1}{4}$ at some point.
\end{proof}
\begin{rmk} The previous proposition shows that the relative tangent bundle of the family of abelian varieties over $\mathbb{L}_{2}$ splits off a line bundle. Dualizing, the same is true for the relative cotangent bundle. We remark that the dual vector to $E_{\bar\alpha}$ is given by $-\tfrac{1}{\sqrt{-N}}\left(\begin{smallmatrix}1\\0\\-1-\alpha\end{smallmatrix}\right)$.
\end{rmk}

%
%
%
%
%
%
%
%
%
%
%
%
%
%
%
%
%
%
%
\section{TAF for $U(1,1;\mathbb{Z}[\zeta_3])$}
%
%
%
%
%
%
%
%
\subsection{Geometry of the Complex Points}\label{Geometry of the Complex Points}
%
%
%

%
%
%
%
\subsubsection*{The Group $U(H_{2}; \mathbb{Z}[\zeta_3])$ }\label{group decomposition in (1,1)}
%
%
%
%
Recall that there is a unique unimodular hermitian lattice of signature $(1,1)$ over the Eisenstein integers, and its isometry group can be identified with $U(H_{2};\ZZ[\zeta_{3}]) \cong U(1,1;\ZZ[\zeta_3])$.  Let $U(H_{2};\ZZ[\zeta_{3}])[\sqrt{-3}]$ be the principal congruence subgroup of level $\sqrt{-3}$, i.e.\ the kernel of the surjective homomorphism
\[
U(H_{2};\ZZ[\zeta_{3}])\to U(H_{2};\ZZ[\zeta_{3}]/(\sqrt{-3}))\cong O(H_{2};\mathbb{F}_{3})\cong C_{2}\times C_{2}.
\]
To relate this group to a more familiar one, note that the left action of
\[
g_{0}=\left(\begin{smallmatrix}(\sqrt{-3})^{-1}&\\&1\end{smallmatrix}\right)\in GL(2;\QQ(\zeta_{3}))
\]
on $\CC P^{1}$ induces an identification $\eta_* ^{-1}\colon\mathbb{L}_{1}\to\mathbb{H}_{1}$, $z_{0}\mapsto z_{0}/\sqrt{-3}$.
\begin{prop} Conjugation by $g_{0}$ yields an isomorphism
\[
U(H_{2};\ZZ[\zeta_{3}])[\sqrt{-3}]\cong \Gamma_{1}(3)\times C_{3},
\]
where $\Gamma_{1}(3)\subset SL(2;\ZZ)$ is a congruence subgroup and $C_{3}$ is cyclic.
\end{prop}
\begin{proof} The homomorphism $\det\colon U(H_{2};\ZZ[\zeta_{3}])[\sqrt{-3}]\to C_{3}$ is surjective, and the central element $\left(\begin{smallmatrix}\zeta_{3}&\\&\zeta_{3}\end{smallmatrix}\right)\in U(H_{2};\ZZ[\zeta_{3}])[\sqrt{-3}]$ is a preimage of a generator of the image, hence $U(H_{2};\ZZ[\zeta_{3}])[\sqrt{-3}]\cong SU(H_{2};\ZZ[\zeta_{3}])[\sqrt{-3}]\times C_{3}$. 
Furthermore, it is well known (and readily verified using the fact that $\ZZ[\zeta_{3}]$ is euclidean with respect to the norm) that there is an isomorphism of groups
\[
\{g\in GL(2;\ZZ[\zeta_{3}]):g^{\dagger}\left(\begin{smallmatrix}0&-1\\1&0\end{smallmatrix}\right)g=\left(\begin{smallmatrix}0&-1\\1&0\end{smallmatrix}\right)\}\cong\{\epsilon h:\epsilon\in(\ZZ[\zeta_{3}])^{\times}, h\in SL(2;\ZZ)\}.
\]
Noting that $\QQ(\zeta_{3})$ is commutative, $g_{0}\left(SU(H_{2};\ZZ[\zeta_{3}])[\sqrt{-3}]\right)g_{0}^{-1}$ is contained in the group on the left-hand side as a subgroup, and, using the right-hand side, this subgroup is readily  identified with $\Gamma_{1}(3)$.
\end{proof}

It now follows easily that a fundamental domain for $g_{0}(U(H_{2};\ZZ[\zeta_{3}]))g_{0}^{-1}$ is given by the strip between $x=-1/2$ and $x=+1/2$ lying above the circle with radius $1/\sqrt{3}$ around 0:

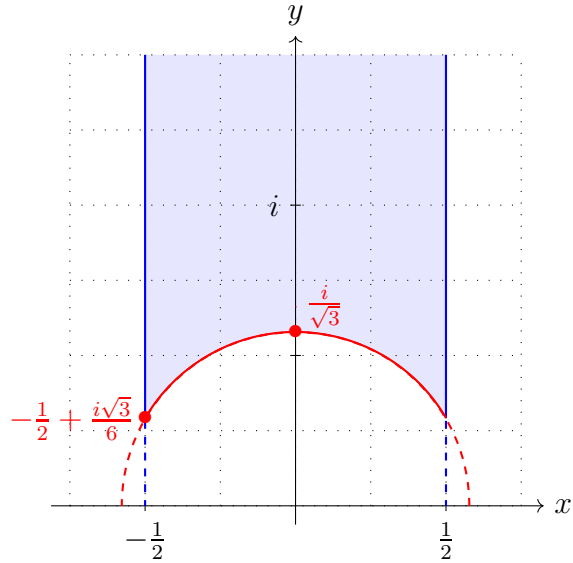
\begin{center}
\begin{tikzpicture}[scale=1.0]
\filldraw[fill=blue!10!white, fill opacity=20] 
 (-2,6) -- (-2,1.17) -- (-2,1.17) arc(149.5:30.5:2.32) -- (2,1.17) -- (2,6) ;
\draw[loosely dotted] (-3,0) grid (3,6);
\draw[->] (-3.25,0) -- (3.3,0) node[right] {$x$};
\foreach \x/\xtext in {-2/{-\frac{1}{2}},  2/{\frac{1}{2}}}
\draw[shift={(\x,0)}] (0pt,2pt) -- (0pt,-2pt) node[below] {$\xtext$};
\draw[->] (0,-0.25) -- (0,6.25) node[above] {$y$};
\foreach \y/\ytext in { 2/{}, 4/i}
    \draw[shift={(0,\y)}] (2pt,0pt) -- (-2pt,0pt) node[left] {$\ytext$};
\draw[thick, blue, dashed]  (-2,1.13) -- (-2,0);
\draw[red] (0,2.65) node[right]{$\frac{i}{\sqrt{3}}$} -- (0,2.65);
\node [red] at (0,2.31) {$\bullet$};
\draw[thick, blue] (2,1.17) -- (2,6);
\draw[thick, blue] (-2,1.17) -- (-2,6);
\draw[thick, blue, dashed]  (2,1.13) -- (2,0);
\node [red] at (-2,1.17) {$\bullet$};
\draw[thick, red] (-2,1.17) node[left]{$-\frac{1}{2}+\frac{i\sqrt{3}}{6}$} -- (-2,1.17);
\draw[red,thick]      (2,1.17) arc (30.5:149.5:2.32);
\draw[red,thick, dashed]      (2.31,0) arc (0:180:2.31);
%

\end{tikzpicture}
\end{center}

%
%
%
%
\subsubsection*{Automorphic Forms}\label{automorphic forms, (1,1)-case}
%
%
%
%
%
Define the intermediate group
 \[
G'=(\Gamma_{1}(3)\cup\Gamma_{1}(3)g_{0}(H_{2})g_{0}^{-1})\times C_{3}.
\]
Hence
\[
 \Gamma_{1}(3)\subset G'\subset G :=g_{0}(U(H_{2};\ZZ[\zeta_{3}]))g_{0}^{-1}.
\]

These three groups act on the upper half plane $\mathbb{H}_{1}$ in the usual fractional linear fashion.
 An automorphic form for $G$ of weight $k$ is a holomorphic function $f\colon\mathbb{H}_{1}\to\CC$ satisfying $$f(g.\tau)=(c\tau+d)^{k}f(\tau)$$ for all $g=\left(\begin{smallmatrix}a&b\\c&d\end{smallmatrix}\right)\in G$. We assign the topological degree $2k$ to a form of weight $k$ and write $M_* ^{\mathbb{C}}(G)$ for the (topologically) graded ring of these forms. 
 We remind the reader that the ring of modular forms for $\Gamma_{1}(3)$ is given by
 \[
 M_{*}^{\CC}(\Gamma_{1}(3))\cong\CC[E_{1},E_{3}],
 \]
 where $E_{1}=1+6\sum_{n}\sum_{d|n}(\frac{d}{3})q^{n}$ and $E_{3}=1-9\sum_{n}\sum_{d|n}(\frac{d}{3})d^{2}q^{n}$ are Eisenstein series of the indicated weight (here, $(\frac{\cdot}{\cdot})$ is the Legendre symbol). 

\begin{prop} The ring of automorphic forms for $G'$ is given by
\[
M_{*}^{\CC}(G')\cong\CC[\kappa,\lambda],
\]
where $\kappa=2E_{3}-E_{1}^{3}$, $\lambda=E_{1}^{6}$ are forms of weight three and six, respectively.
\end{prop}
\begin{proof}
Clearly, the ring of automorphic forms for $G'$ consists of the modular forms for $\Gamma_{1}(3)$ of weight divisible by three which are invariant under the involution induced by $g_{0}(H_{2})g_{0}^{-1}$. Now recall that for a symmetric positive definite form $S$ of rank $n$ we have the following transformation rule for its theta series 
\[
\theta(S^{-1},-\tfrac{1}{\tau})=\left(\tfrac{\tau}{i}\right)^{n/2}(\det S)^{1/2}\theta(S,\tau).
\]
Since $E_{1}(\tau)$ is the theta series of the $A_{2}$ root lattice, we have $E_{1}(-\frac{1}{3\tau})=-\sqrt{-3}\tau E_{1}(\tau)$. Next, let $E_{4}(\tau)=1+240\sum\sum_{d|n}d^{3}\ q^{n}$ be the usual Eisenstein series of weight four, and let $e_4(\tau):=9E_{4}(3\tau)-E_{4}(\tau)$; then we have
\[
\left(\tfrac{1}{\sqrt{-3}\tau}\right)^{4}e_4(-\tfrac{1}{3\tau})=-e_4(\tau).
\]
Since $e_4(\tau)$ is a modular form for $\Gamma_{1}(3)$, comparing $q$--expansions shows $$e_4(\tau)=8E_{1}(2E_{3}-E_{1}^{3}).$$ Therefore we conclude that $\kappa=2E_{3}-E_{1}^{3}$ spans the $+1$ eigenspace of the involution
\begin{equation}\label{an involution}
\star\colon M_{6}^{\CC}(\Gamma_{1}(3))\to M_{6}^{\CC}(\Gamma_{1}(3)),\ m(\tau)\mapsto\left(\sqrt{-3}\tau\right)^{-3}m\left(-\tfrac{1}{3\tau}\right),
\end{equation}
while $E_{1}^{3}$ spans the $-1$ eigenspace, establishing the claim.
\end{proof}
Each automorphic form for $G$ admits a Fourier expansion at the (unique) cusp $i\infty$; for all rings $\mathbb{Z} \subseteq R \subseteq \mathbb{C}$ we denote by $M^R_*(G)$ the ring of automorphic forms for $G$ with expansions in $R[\![q]\!]$, where $q=e^{2\pi i\tau}$.
\begin{prop} The ring of automorphic forms for $G$ expanding integrally at the cusp $i\infty$ is given by
\[
M_{*}^{\ZZ}(G)\cong\ZZ[E_{1}^{6},\Delta_{6}],
\]
where $\Delta_{6}=E_{3}(E_{1}^{3}-E_{3})/27$ is the normalized cusp form of weight six.
\end{prop}
\begin{proof} Combined with the coset decomposition $G=G'\cup G'\left(\begin{smallmatrix}-1&\\&-1\end{smallmatrix}\right)$, the previous proposition implies $M_{*}^{\CC}(G)\cong\CC[\kappa^{2},\lambda]$. We also have $\kappa^{2}=\lambda-2^{2}\cdot3^{3}\cdot\Delta_{6}$, hence $M_{*}^{\CC}(G)\cong\CC[\lambda,\Delta_{6}]$.
 Since $\Delta_{6}=q+O(q^{2})\in\ZZ[\![q]\!]$ and $E_{1}^{6}=1+36q+O(q^{2})\in\ZZ[\![q]\!]$, the claim follows via induction.
\end{proof}

We remark that the behavior under the transformation induced by the element $g_{0}(H_{2})g_{0}^{-1}\in G'$ implies
\[
E_{1}((-\tfrac{1}{2}-\tfrac{\sqrt{-3}}{2})/\sqrt{-3})=0,\ \kappa(\sqrt{-3}/3)=0,
\]
and by construction, the form $\Delta_{6}$ vanishes at the cusp $i\infty$. Up to $G'$--equivalence,  these are the only zeros:

\begin{prop} 		\label{ka sechstel}

Let $f\colon\mathbb{H}_1\to\CC$ be a meromorphic function satisfying  $f(g.\tau)=(c\tau+d)^{k}f(\tau)$ for all $g=\left(\begin{smallmatrix}a&b\\c&d\end{smallmatrix}\right)\in G'$. Then the orders of its poles and zeros satisfy
\[
{\rm ord}(f,i\infty)+\textstyle{\frac{1}{2}}{\rm ord}(f,\textstyle{\frac{\sqrt{-3}}{3}})+\textstyle{\frac{1}{6}}{\rm ord}(f,\textstyle{-\frac{1}{2}+\frac{\sqrt{-3}}{6}})+\sum_{p}{\rm ord}(f,p)=\textstyle{\frac{k}{6}},
\]
where the sum is taken over a set of representatives mod $G'$.
\end{prop}
\begin{proof} This proceeds as in the proof of the analogous formula for the group $SL(2;\ZZ)$: More precisely, we have
$$\sum_{p}{\rm ord}(f,p)=\frac{1}{2\pi i}\int_{C}\frac{f'(z)}{f(z)}dz,$$ where the contour $C$ goes around a fundamental domain for $G'$  excluding the points $\pm\frac{1}{2}+\frac{\sqrt{-3}}{6}$ and $\frac{\sqrt{-3}}{3}$. Noting that the enclosed angle at $-\frac{1}{2}+\frac{\sqrt{-3}}{6}$ is $\pi/6$ and the angle from $-\tfrac{1}{2}+\tfrac{\sqrt{-3}}{6}$ to $\frac{\sqrt{-3}}{3}$ along the circle around $0$ is ${\pi}/{3}$,  the claim follows.
\end{proof}

The following proposition tells us that there is only a single point in the ball quotient where the corresponding abelian two-fold is not a Jacobian. It is the product of two Jacobians, though, as product of two elliptic curves with complex multiplication by the Eisenstein integers.

\begin{prop}\label{indecomposable period}
 If a period matrix in the image of $\iota_*$ is diagonalizable by an integral symplectic transformation, then it is $G$--equivalent to the point $(\iota\circ\eta)_*(-\tfrac{1}{2}+\tfrac{\sqrt{-3}}{6})$.
\end{prop}

\begin{proof}
If $\Omega$ is diagonalizable by the fractional linear action of an element in $Sp(2;\ZZ)$, then one of the ten even theta constants vanishes. Now let $\Pi(\Omega)$ be the product of all ten theta constants raised to the eighth power. Then the transformation law of theta functions implies
\[
\Pi\big((A\Omega+B)(C\Omega+D)^{-1}\big)=(\det(C\Omega+D))^{40}\Pi(\Omega)
\]
for all $\left(\begin{smallmatrix}A&B\\C&D\end{smallmatrix}\right)\in Sp(2;\ZZ)$.
We have
\[
\Theta\left[\begin{smallmatrix}0&0\\0&0\end{smallmatrix}\right]((\iota\circ\eta)_{*}\tau)=E_{1}(\tau).
\]
The remaining nine even Theta characteristics constitute a single orbit under the full group $G$; however, discarding the action of the roots of unity, they decompose into three orbits, and the $q$--expansion reveals
\[
(\iota\circ\eta)^{*}\left(\Theta\left[\begin{smallmatrix}0&0\\1&1\end{smallmatrix}\right]\Theta\left[\begin{smallmatrix}1&1\\1&1\end{smallmatrix}\right]\Theta\left[\begin{smallmatrix}1&1\\0&0\end{smallmatrix}\right]\right)^{8}=2^{16}\Delta_{6}^{4},
\]
where, for notational convenience, the half-integral characteristic is taken in $(\ZZ/2\ZZ)^{4}$,
and one obtains the same value for the other two orbits. Therefore, the product of all ten even Theta constants raised to the eight power pulls back to yield $2^{48}E_{1}^{8}\Delta_{6}^{12}$; since $\Delta_{6}$ is the normalized cusp form of weight six, the claim follows.
\end{proof}

\begin{prop}
The function
\[
j_G \colon\left(\mathbb{H}_1\cup\QQ P^{1}\right)/G\to\CC\cup\{\infty\}\cong\CC P^{1}
\] given by 
$$ j_G(\tau) = \frac{\lambda}{(-2\kappa)^{2}}(\tau)=\frac{E_{1}^{6}}{4(E_{1}^{6}-4E_{3}(E_{1}^{3}-E_{3}))}(\tau)$$
is a bijection.
\end{prop}

\begin{proof} The function $j_G$ has weight zero and a pole in $\frac{\sqrt{-3}}{3}$. This remains true for any translate $j_G-w$, where $w\in\CC$, hence $j_G-w$ has a unique zero in $\left(\mathbb{H}_1\cup\QQ P^{1}\right)/G$.
\end{proof}
\begin{rmk} $j_G(i\infty)=\frac{1}{4}$, $j_G(-\frac{1}{2} + \frac{\sqrt{-3}}{6}) = 0$ and $j_G(\frac{\sqrt{-3}}{3}) = \infty.$
\end{rmk}

%
%
%
%
%
%
%
\subsection{TAF via Hyperelliptic Curves}
%
%
%
%
Recall that a smooth algebraic curve $C$ of genus $n$ has an associated $n$-dimensional principally polarized abelian variety ${\rm Jac}(C)={\rm Pic}^0(C)$, called its Jacobian variety.
Further,  a (smooth) curve $C$ of genus $n$ is hyperelliptic if there is a morphism $C\to \mathbb{P}^{1}$ of degree two. These curves have an affine model given by an equation 
\[
C': \;y^{2}=F_{d}(x),
\]
where $F_{d}$ is a polynomial of degree $2n+1$ or $2n+2$ without multiple roots. Conversely, every such equation yields a model of a hyperelliptic curve of genus $n$ with a unique singularity at the point at infinity that corresponds to one or two points in a regular model, depending on whether the degree $d$ is odd or even, respectively. It is standard practice to model such curves by glueing together two non-singular affine curves $$C= C' \cup C'',$$ one given by the affine model $C'$ above and the other 
\[
C'': v^2 = u^{2n+2}F_{d}({u}^{-1})
\]
by changing coordinates via $x=u^{-1}$ and $y=\tfrac{v}{u^{n+1}}$. 
In the latter model the point(s) at infinity correspond(s) to the point(s) where $u=0$.
\begin{rmk} For an explicit embedding of  the Jacobian of a hyperelliptic curve of genus two in general sextic form, $y^{2}=F_{6}(x)$, into $\mathbb{P}^{15}$ and a description of the corresponding two-dimensional formal group law, see \cite{Flynn:1993aa}.
\end{rmk}
%
%
%
%
\subsubsection*{A Family of Hyperelliptic Curves}		\label{SecFamHypCurves}
%
%
%
%
It is known that every curve of genus two is hyperelliptic. Clearly, every hyperelliptic curve admits a canonical involution, which, in terms of the equation given above, is given by 
\[
(x,y)\mapsto (x,-y).
\]
Long ago, Bolza determined the curves of genus two with larger automorphism group; relevant to us is \cite[Case IV]{Bolza:1887aa}, i.e.\ hyperelliptic curves of genus two with affine model
\begin{equation}\label{Bolza's family}
y^{2}=x^{6}+ax^{3}+1.
\end{equation}
Such curves admit an obvious automorphism of order three, generated by
\begin{equation}\label{C_{3}--action on the curve}
\mu\colon(x,y)\mapsto(\zeta_{3}x,y);
\end{equation}
for a suitable basis of the first integral homology, the induced action on the extended period matrix is given by the matrix $\sigma(t)J_{N}\sigma(t^{-1})$, which in turn implies that the entries of the normalized period matrix satisfy $\Omega_{11}=\Omega_{22}=2\Omega_{12}$,  i.e.\ $(\Omega_{ij})$ is in the image of $\iota_{*}\colon\mathbb{L}_{1}\to\mathbb{S}_{2}$.

\begin{rmk} The curve with affine model \eqref{Bolza's family} also admits the involution $(x,y)\mapsto(1/x,y/x^{3})$. Taking into account \eqref{C_{3}--action on the curve} and the hyperelliptic involution, we therefore conclude that the automorphism group of this curve contains a subgroup isomorphic to the dihedral group with twelve elements.
\end{rmk}

As a modification of \eqref{Bolza's family}, consider the family
\begin{equation}\label{modular version of Bolza's family}
C' : y^{2}=x^{6}-2\kappa x^{3}+\lambda,
\end{equation}
parametrized using the automorphic forms $\kappa,\lambda\in M_{*}^{\CC}(G')$ introduced earlier. Clearly, smoothness (of the affine model) is equivalent to the non-vanishing of the discriminant of the polynomial $(x^{6}-2\kappa x^{3}+\lambda)$,
\[
\Delta_C = 2^{6}\cdot 3^{6}\cdot\lambda^{2}(\lambda-\kappa^{2})^{2}=2^{10}\cdot3^{12}\cdot\lambda^2\cdot\Delta_6^2.
\]
In this case, the differentials $dx/y$ and $xdx/y$ are holomorphic and span the eigenspaces of the induced $C_{3}$--action. In particular, the second differential is distinguished by $\mu^{*}(xdx/y)=\zeta_{3}^{-1}xdx/y$.

Over the complex numbers, the family \eqref{modular version of Bolza's family} is universal:
\begin{prop}
Evaluating the automorphic forms $\kappa$ and $\lambda$ on the upper half plane induces a bijection between the points of $\mathbb{H}_{1}/G\setminus\{[-\tfrac{1}{2}+\tfrac{\sqrt{-3}}{6}]\}$ and the set of $\CC$--isomorphism classes of curves of genus two whose automorphism group contains the dihedral group of order twelve.
\end{prop}

\begin{proof}
According to \cite[Proposition 2.2]{Cardona:2007aa}, the family $Y^{2}=X^{6}+X^{3}+j$, where $j\in\CC\setminus\{0,\frac{1}{4},-\frac{1}{50}\}$, gives all the curves (over $\CC$) with automorphism group isomorphic to the dihedral group; adding the two curves with larger symmetry group and the singular curve corresponding to $j=0$, we arrive at $\mathbb{H}_1/G$, and adjoining the cusp (i.e.\  the degenerate curve $j=\frac{1}{4}$) yields the compact space $\CC P^{1}$.
\end{proof}

\bigskip

Note that when we change coordinates and work with $$ C'' : v^2 = 1-2\kappa u^3 + \lambda u^6,$$
the action \eqref{C_{3}--action on the curve} becomes $(u,v)\mapsto(\zeta_{3}^{-1}u,v)$.
Furthermore, the distinguished differential $-xdx/y$ is mapped to
\begin{align} \label{HyperDifferential} 
\frac{du}{v} = \left(1-2\kappa u^3 + \lambda u^6\right)^{-1/2}du,
\end{align}
which for $\lambda=0$ reduces to the invariant differential of an elliptic curve with (complex conjugate) action of $\ZZ[\zeta_{3}]$.
%
%
%
%
\subsubsection*{The $p$-complete Point of View}		\label{SecHyperPcomplete}
%
%
%
%
We outline how the formal group of the Jacobian of the curve splits over the $p$-completion of its ring of definition $R$.
Recall that the curve, its Jacobian, thus all differentials and in particular the splitting of the Lie algebra of the Jacobian are defined over $$R= \mathbb{Z}[\zeta_3][\tfrac{1}{6}][ \kappa, \lambda, (\lambda^{2}\Delta_{6}^{2})^{-1}].$$ Inverting ${6}$ and  $(\lambda\Delta_{6})^{2}$ ensures smoothness of the curve, whereas adjoining $\zeta_3$ is necessary when considering the $\mathbb{Z}[\zeta_3]$-actions.

\begin{lem}

$\mathbb{G}_{{\rm Jac}(C)}$ is  a $\ZZ_p$-module.

\end{lem}

\begin{proof}

Every $p$-adic number $c$ can be displayed as limit  $$ c= {\rm lim}_n \; c^{(n)} $$ over integers
$c^{(n)}\in \ZZ$. For each $c^{(n)}$ the addition $[c^{(n)}]_{\mathbb{G}_{{\rm Jac}(C)}}(x) = \sum_{i=1}^\infty c_i^{(n)} x^i$ for a point $x\in \mathbb{G}_{{\rm Jac}(C)}$ is well defined. 
Let us define $$ [c]_{\mathbb{G}_{{\rm Jac}(C)}}(x) = \sum_{i=1}^\infty \lim_{n} c_i^{(n)} x^i.$$ This limit exists because from 
\begin{align*}
c^{(n)} - c^{(n-1)} = b^{(n)}p^n \quad \Leftrightarrow \quad c^{(n)} \equiv c^{(n-1)} \mod p^n
\end{align*}
it follows that 
\begin{align*}
[c^{(n)}](x)  =  [c^{(n-1)}](x)  + [p^n]([b^{(n)}](x)).
\end{align*}
Giving $p$ and $x$ degree $1$, one uses an easy inductive argument to see that the second term on the right hand side has degree greater or equal to $n+1$. This show that in the limit $\mathbb{G}_{{\rm Jac}(C)}$ is a $\ZZ_p$-module.

\end{proof}

Note that the statement of the lemma implies that  $$\ZZ_p\hookrightarrow {\rm End}(\mathbb{G}_{{\rm Jac}(C)}).$$

\begin{lem}

$\ZZ_p$ is central in ${\rm End}(\mathbb{G}_{{\rm Jac}(C)})$.

\end{lem}

\begin{proof}

Let again $c= \lim_n c^{(n)} \in \ZZ_p$ and $\psi \in {\rm End}(\mathbb{G}_{{\rm Jac}(C)})$. Since $c^{(n)} \in \ZZ$, we have $$ \psi ([c^{(n)}](x)) = [c^{(n)}](\psi(x)).$$ Taking limits proves the claim.

\end{proof}

The two preceding lemmas imply the following.

\begin{prop}

The action $R \hookrightarrow {\rm End}(\mathbb{G}_{{\rm Jac}(C)})$ extends to $R\otimes \ZZ_p$, i.e.\ \begin{center}
\begin{tikzcd}

R \ar[hook]{rr} \ar[hook]{dr} & &  {\rm End}(\mathbb{G}_{{\rm Jac}(C)}) \\  & R\otimes \ZZ_p \ar[hook, dashed]{ur} & 

\end{tikzcd}
\end{center}
is a commutative diagram of rings.
\qed
\end{prop}

\begin{cor}		\label{CorPcompleteSplittingTheFormalGroup}

For (rational) primes which split $p=u\bar{u}$ in the imaginary quadratic number field $\mathbb{K}= \mathbb{Z}[\zeta_3]\otimes \QQ$, the formal group $\mathbb{G}_{{\rm Jac}(C)}$ splits $$\mathbb{G}_{{\rm Jac}(C)} \cong \mathbb{G}^{+} \times \mathbb{G}^{-}.$$ The dimension of the formal group $\mathbb{G}^{-}$ is one.

\end{cor}

\begin{proof}

We find idempotents in ${\rm End}(\mathbb{G}_{{\rm Jac}(C)})$, coming from the factorization of $$1= (e_1 , e_2)\in \mathbb{Z}[\zeta_3]\otimes \ZZ_p \cong \ZZ_p \times \ZZ_p,$$ giving rise to the indicated splitting. 
The dimensions of the formal groups are determined by the dimension of the Lie algebras, which are by duality determined by the $\zeta_3$-action on the space differentials. As described above the action distinguishes the one (in fact both, as there are only two) differential which was acted on by conjugation.
\end{proof}

In what follows we use the presentation of the curve to not only refine the above argument to a $p$-local statement, but also give an explicit description of the associated $p$-local genus.

%
%
%
\subsubsection*{Formal Parameters}
%
%
%
%

As explained above the curve $C$ has two (smooth) points at infinity, since the degree of the describing polynomial is even. Let us consider the point corresponding to $P=(0,1)$ on $C''$. This point is a fix point of the $\zeta_3$-action and only the distinguished differential is non-vanishing. This implies that the Torelli map identifies $ T_PC'' $ $C_3$-equivariantly with the $1$-dimensional summand ${\rm Lie}^-({\rm Jac}(C''))$.

Finding a local parameter at the point at infinity of $C$ amounts to finding a local parameter at $P$ on $C''$. 

Let $R= \mathbb{Z}[\zeta_3][\tfrac{1}{6}][ \kappa, \lambda, (\lambda^{2}\Delta_{6}^{2})^{-1}]$, then the coordinate ring of $C'' $ is $$\mathcal{O}_{C''} = \raisebox{1ex}{\textit{R}[\textit{u,v}]\big/}\!\! \left( v^2-(1-2\kappa u^3 + \lambda u^6)\right). $$ The point $P$ is represented by the ideal $I=(u,v-1)$. Since $I$ is prime, the localization $\mathcal{O}_{C'', (P)}$ is a local ring with maximal ideal $\mathfrak{m}=I$. Now, since 
\begin{align}
v^2&=1-2\kappa u^3 + \lambda u^6 \nonumber \\ 
\Leftrightarrow \; (v-1)(v+1) &= u^3\cdot p(u^3),  \label{LocParU(1,1)}
\end{align}
where $p\in \mathbb{Z}[x]$ and $(v+1)\not \in \mathfrak{m}$ is a unit, we find that  $$(v-1) \subset (u).$$ So the maximal ideal is given by $(u)$ and thus $u$ is a uniformizer as desired.
%
%
%
%
\subsubsection*{A Formal Logarithm}			\label{SectionU11FormalLog}
%
%
%
%

We define a rational genus
\[
\varphi^{L}\colon MU_{*}^{\QQ}\to M_{*}^{\QQ}(G')\cong\QQ[\kappa,\lambda]
\]
by
\[
\text{log}_{\varphi^{L}}'(u)=(1-2\kappa u^{3}+\lambda u^{6})^{-{1}/{2}}. 
\]
In other words, we define the logarithm of our genus to be the integral $\int \frac{du}{v}$ of the distinguished differential. Recall that 
\[
(1-2\kappa t+\lambda t^{2})^{-1/2}=\sum_{k\geq0}P_{k}(\kappa,\lambda)t^{k}.
\]
is the generating function of the homogeneous Legendre polynomials $P_{k}$. Consequently, the value of $\varphi^{L}$ on the complex projective spaces can be expressed in terms of Legendre polynomials.
%

We have the following integrality statement:
\begin{thm}		\label{ThmHyperIntegrality}
For each prime $p\equiv1\!\mod3$, the image of $BP_*$ under $\varphi^{L}$ is contained in the subring of $p$-local automorphic forms:
\[
\varphi^L(BP_*) \subset M^{\mathbb{Z}_{(p)}}_*(G) \cong \mathbb{Z}_{(p)}[\kappa^2, \lambda] \subset \mathbb{Z}_{(p)}[\kappa, \lambda].
\]
\end{thm}

\begin{proof}
We prove the claim making use of an arithmetic fracture square
\begin{center}
\begin{tikzcd}
\mathbb{Z}_{(p)} \ar{d} \ar{r} & \mathbb{Z}_p \ar{d} \\ \mathbb{Q} \ar{r} & \mathbb{Q}_p.
\end{tikzcd}
\end{center}
On the one hand, we have the formal group law over $\QQ[\kappa,\lambda]$ determined by $\varphi^{L}$. On the other hand, we have the group law over $\mathbb{Z}[\kappa,\lambda, \Delta_C^{-1}]^\wedge_p$ (which already contains $\tfrac{1}{6}$ and $\zeta_3$) arising from the split summand of the Jacobians of the family \eqref{modular version of Bolza's family}. Recall that the point $P=(0,1)\in C''$ is a fixed point of the $C_{3}$--action, and that we have an equivariant identification $$T_{P}C''\cong {\rm Lie}^-({\rm Jac}(C'')).$$ 
Choosing a local parameter $u$, as above, at the point  $P\in C''$, the parametrization $(u, v(u))$ allows us to pull back the $1$-dimensional summand of $\mathbb{G}_{{\rm Jac}(C)}$ to a formal neighborhood of the point $P$ in $C''$. In particular, the pullback of the formal differential of the $1$-dimensional summand agrees with the coordinate expression of \eqref{HyperDifferential} on the nose. This shows that the formal group law is already  defined over $\ZZ_{(p)}[\kappa,\lambda]$.

Moreover, since $p\equiv1\!\mod3$ and $p$ prime implies $p\equiv1\!\mod6$, we have that $\varphi^{L}(BP_{*})$ is contained in the smaller subring $\ZZ_{(p)}[\kappa^{2},\lambda]\subset\ZZ_{(p)}[\kappa,\lambda]$ concentrated in degrees divisible by twelve.
\end{proof}
\begin{rmk}		\label{PropGenusDegEll} 
In the only elliptic point $[\tau]=[-\tfrac{1}{2}+\tfrac{\sqrt{-3}}{6}]\in\mathbb{H}_{1}/G$, corresponding to the product ppav $E\times\bar E$, the formal group law associated to the genus $\varphi^L$ restricts to an elliptic group law.
\end{rmk}

%
%
%
%
\subsubsection*{$p$-local TAF}
%
%
%
%
Having established integrality, we can give explicit descriptions of $p$--local TAF theories by verifying Landweber's criterion \cite{Landweber:1976lo}. For a fixed prime $p\equiv1\!\mod3$, let $v_{k}$ denote the image of the $k^{th}$ Hazewinkel generator  \cite[Appendix A2]{Ravenel:2004xh} under $\varphi^{L}$; in particular, we have
\[
v_{1}=P_{\frac{p-1}{3}}(\kappa,\lambda),\quad v_{2}=\frac{1}{p}\left(P_{\frac{p^{2}-1}{3}}(\kappa,\lambda)-P_{\frac{p-1}{3}}^{p+1}(\kappa,\lambda)\right).
\]
Recall that $\Delta_{6}$ is the normalized cusp form of weight six.
\begin{cor}\label{height two}
Let $p=7$. Then $\varphi^{L}$ gives $M_{*}^{\ZZ_{(7)}}(G)[\Delta_{6}^{-1}]$ the structure of a Landweber exact $BP_*$-algebra of height two. In particular, the functors
\[
TAF^{U(1,1;\ZZ[\zeta_3])}_{(7),*}(\cdot):=BP_{*}(\cdot)\otimes_{\varphi}M^{\ZZ_{(7)}}_{*}(G)[\Delta_{6}^{-1}]
\]
define a homology theory.
\end{cor}
\begin{proof}
Clearly, multiplication by 7 is injective. Since $P_{2}(x,1)=\frac{1}{2}(3x^{2}-1)$, it follows that
\[
v_{1}=E_{1}^{6}-6(E_{1}^{3}E_{3}-E_{3}^{2})\equiv \lambda-\Delta_{6}\mod7,
\]
which is regular on $\mathbb{F}_{7}[\lambda,\Delta_{6}]$. Now from $(P_{16}(\frac{1}{\sqrt3},1)-P_{2}^{8}(\frac{1}{\sqrt3},1))/7=-\frac{19\cdot 113}{2^{7}\cdot 3^{6}}$ and  $-\frac{19\cdot 113}{2^{7}\cdot 3^{6}}\equiv1\mod7$ we conclude
\[
v_{2}\equiv-\frac{19\cdot 113}{2^{7}\cdot 3^{6}}\lambda^{8}\equiv\Delta_{6}^{8}\mod(7,v_{1});
\]
since $\Delta_{6}$ is invertible, the claim follows.
\end{proof}

\begin{cor}\label{height two p=13}
Let $p=13$. Then $\varphi^{L}$ gives $M_{*}^{\ZZ_{(13)}}(G)[\Delta_{6}^{-1}]$ the structure of a Landweber exact $BP_*$-algebra of height two. In particular, the functors
\[
TAF^{U(1,1;\ZZ[\zeta_3])}_{(13),*}(\cdot):=BP_{*}(\cdot)\otimes_{\varphi}M^{\ZZ_{(13)}}_{*}(G)[\Delta_{6}^{-1}]
\]
define a homology theory.
\end{cor}
\begin{proof}
Again, multiplication by 13 is injective and 
\[
v_{1}\equiv 6\kappa^4 + 6\kappa^2 \lambda + 2\lambda^2 \mod13
\]
is regular on $\mathbb{F}_{13}[\lambda,\Delta_{6}]$. Since further 
\[
v_{2}\equiv \lambda^{28} \equiv \Delta_6^{28} \mod(13, v_1)
\]
the claim follows.
\end{proof}

\begin{rmk}\label{known coefficients}
We should remark that this is closely related to \cite[Theorem 7.4]{Behrens:2011aa} in that, up to scaling, $E_1^6$ and $\Delta_6$ coincide with their $a_1^6$ and $D$.
\end{rmk}

\begin{rmk}
Inverting the cusp form $\Delta_6$ amounts to excluding the cusp. Since $$\Delta_C = 2^{10}\cdot3^{12}\cdot\lambda^2\cdot\Delta_6^2,$$ we could also have inverted $\lambda$ or both factors. Inverting $\lambda$ excludes the elliptic point. The same argument shows that $\ZZ_{(7)}[\kappa^{\pm1},\lambda^{\pm1}]$ is a Landweber exact algebra of height two.
\end{rmk}

\begin{rmk} Let us emphasize the value of hands-on computations as presented here for the primes $p=7,13$. These computations let us recognize that it is sufficient to invert less than the discriminant $\Delta_{C}$ of the curves; if it is inverted, then, as the variety satisfies a universal deformation criterion and all the height loci are non-empty, 
it is apparent from the general theory presented in \cite{Behrens:2010aa} that the sequence $(p,v_{1},v_{2})$ is regular for all primes $p\equiv1\!\mod3$. 
\end{rmk}

%
%
%
%
%
%
%
%
%
%
\section{Height $3$ TAF via Picard Curves}		\label{SecTAF U(2,1)} 
%
%
%
%
%
%
%
%
%
%
%
%
%
%
%
\subsection{Picard Curves}
%
%
%
%
We start by considering Shiga's curves \cite{Shiga:1988aa}, which are a special case of Picard curves; namely the case where all roots of the defining polynomial are in the base ring. This amounts to considering the distinguished level structure of the usual lattice in the unitary group, i.e.\ the arithmetic subgroup $\Gamma' = U(H_3;\mathbb{Z}[\zeta_3])[\sqrt{-3}]\cong U(2,1;\mathbb{Z}[\zeta_3])[\sqrt{-3}]$.

%
%
%
%
\subsubsection*{Shiga's Family}				\label{SectionShigasFamily}
%
%
%
%
Consider a plane curve defined by
\begin{equation}\label{Shiga family}
C\colon y^{3}=x(x-\xi_{0})(x-\xi_{1})(x-\xi_{2}).
\end{equation}
Passing to the projective closure adds one point at infinity, which is easily seen to be smooth. On the other hand, smoothness of the affine part is equivalent to the non-vanishing of the discriminant
\[
\Delta_{C}=\left(\xi_{0}\xi_{1}\xi_{2}(\xi_{0}-\xi_{1})(\xi_{1}-\xi_{2})(\xi_{2}-\xi_{0})\right)^{2}.
\]
In the smooth case, i.e.\ if $\Delta_{C}\neq0$, the curve \eqref{Shiga family} has genus three, and a basis for the space of holomorphic differentials is given by
\[
\frac{dx}{y^{2}},\ \frac{xdx}{y^{2}},\ \frac{dx}{y}. 
\]
There is an obvious action of the cyclic group $C_{3}$ on the curve, generated by
\[
(x,y)\mapsto(x,\zeta_{3}y),
\]
and the induced action on the space of differentials singles out the space spanned by $dx/y$. As a consequence, the Jacobians of these curves are ppavs with  $\ZZ[\zeta_{3}]$--endomorphisms of type $(2,1)$. 

Let us be a bit more precise. As can be deduced directly from their definition, smooth Shiga-Picard curves are parametrized by $$\Xi =  \{\left[\begin{smallmatrix}\xi_{0}\\\xi_{1}\\ \xi_{2}\end{smallmatrix}\right]\in\mathbb{P}^{2}:\Delta_{C}\neq0\}.$$ This can be extended to the case where we allow arbitrary Picard curves (cf. Section \ref{SectionDegenerations}), viz.\ where the describing polynomial has multiple roots, which are then parametrized by $\mathbb{P}^2$ itself.

Picard \cite{Picard:1883aa} showed that there is a period mapping $$\Phi \colon \Xi \to \mathbb{S}_3$$ to the Siegel upper half space, with dense image in the subvariety $\mathbb{C}H^2\subset \mathbb{S}_3$. This map induces a biholomorphic correspondence 
\begin{align}	\label{PeriodIso} 
\Phi_* \colon  \mathbb{P}^2  \xrightarrow{\cong}(\mathbb{C}H^2/\Gamma')^*,
\end{align}
between $\mathbb{P}^2$ and the Satake--Baily--Borel compactification $(\mathbb{C}H^2/\Gamma')^*$ of the quotient of $\mathbb{C}H^2$ by the arithmetic subgroup $\Gamma' \cong U(2,1;\mathbb{Z}[\zeta_3])[\sqrt{-3}]$ corresponding to the distinguished level structure over the Eisenstein integers.

Further, he showed that the inverse, now called (Picard) modular function, of this period mapping is a single valued automorphic function on $\mathbb{C}H^2$, that can be represented via Riemann theta functions. Automorphicity of the modular function is relative to the group $\Gamma'$. A holomorphic function $\phi\colon\mathbb{L}_{2}\to\CC$ satisfying
\begin{equation}\label{transformation property}
\phi(\gamma.\left(\begin{smallmatrix}z_{0}\\z_{1}\end{smallmatrix}\right))=(\gamma_{31}z_{0}+\gamma_{32}z_{1}+\gamma_{33})^{3k}\phi\left(\begin{smallmatrix}z_{0}\\z_{1}\end{smallmatrix}\right)
\end{equation}
for all $\gamma=(\gamma_{ij})\in\Gamma'$ will be called {\em automorphic form for $\Gamma'$ of topological degree $6k$}; we denote the finite-dimensional $\CC$--vector space of such forms by $M_{6k}^{\CC}(\Gamma')$, and write $M_{*}^{\CC}(\Gamma')=\bigoplus_{k\in\NN_{0}}M_{6k}^{\CC}(\Gamma')$ for the corresponding graded ring. 

In \cite{Shiga:1988aa} Shiga reviews work of Picard and improves it by showing how the modular function can be expressed using theta constants. More precisely, he shows that the inverse to \eqref{PeriodIso} is induced by
$$ \left(\begin{smallmatrix}z_{0}\\z_{1}\end{smallmatrix}\right) \mapsto \left[\begin{smallmatrix} \phi_{0}\\ \phi_{1}\\ \phi_{2}\end{smallmatrix}\right] =: \left[\begin{smallmatrix}\xi_{0}\\\xi_{1}\\ \xi_{2}\end{smallmatrix}\right] \in \mathbb{P}^2,$$ where $$ \phi_k\left(\begin{smallmatrix}z_{0}\\z_{1}\end{smallmatrix}\right) = \theta^3\left[ \begin{smallmatrix} 0 & 1/6 & 0 \\ k/3 & 1/6 & k/3 \end{smallmatrix} \right] (0, \tilde\iota_*\left(\begin{smallmatrix}z_{0}\\z_{1}\end{smallmatrix}\right) )$$ and $\tilde\iota_*$ is essentially the twisted embedding of Section \ref{SectionTwistedEmbedding}. 

\begin{rmk}\label{Picard's period matrix}
Explicitly, let
\[
\left(\begin{smallmatrix}A&B\\C&D\end{smallmatrix}\right)=\left(\begin{smallmatrix}&&1&0&&\\&0&&&-1&\\-1&&&&&0\\0&&&&&1\\&1&&&0&\\&&0&-1&&\end{smallmatrix}\right)\in Sp(3;\ZZ);
\]
applying the corresponding fractional linear transformation to $\iota_{*}\left(\begin{smallmatrix}z_{0}\\z_{1}\end{smallmatrix}\right)$ and using $\zeta_{3}+\zeta_{3}^{2}+1=0$, we get the embedding
\[
\begin{pmatrix}z_{0}\\z_{1}\end{pmatrix}\mapsto\tilde\iota_{*}\begin{pmatrix}z_{0}\\z_{1}\end{pmatrix}=\begin{pmatrix}\frac{\zeta_{3} z_{1}^{2}+2z_{0}}{\sqrt{-3}}& \zeta_{3}^{2}z_{1}&\frac{\zeta_{3}^{2} z_{1}^{2}-z_{0}}{\sqrt{-3}}\\ \zeta_{3}^{2}z_{1}&-\zeta_{3}^{2}& z_{1}\\\frac{\zeta_{3}^{2} z_{1}^{2}-z_{0}}{\sqrt{-3}}& z_{1}&\frac{z_{1}^{2}+2z_{0}}{\sqrt{-3}}\end{pmatrix}.
\]
\end{rmk}

Shiga also characterizes $\xi_{0},\xi_{1}, \xi_{2} $ as automorphic forms of $\mathbb{C}H^2$ relative to $\Gamma'$ and determines the graded ring of (complex) automorphic forms to be
\[
M_{*}^{\CC}(\Gamma')\cong\CC[\phi_{0},\phi_{1},\phi_{2}]
\]
where each $\phi_{i}$ has topological degree six. 
Note that the $\tfrac{\phi_{1}}{\phi_{0}}, \tfrac{\phi_{2}}{\phi_{0}}$ are meromorphic automorphic functions which generate the function field of $(\mathbb{C}H^2/\Gamma')^*$.

\smallskip

Thus, the curves \eqref{Shiga family} can be thought of as an analytic family essentially parametrized by $\CC P^{2}$, and the locus of smooth curves corresponds to 
 the complement of the union of six hyperplanes. On the other hand, these smooth curves can be treated as a family over the ring of automorphic forms, adjoining the inverse of the discriminant to ensure smoothness.
%
%
%
%
\subsubsection*{Degenerations}			\label{SectionDegenerations}
%
%
%
%
The parameter space $\Xi$ of smooth Picard curves of genus three (with markings) can be identified with the configuration space of five ordered distinct points in $\mathbb{P}^1$ up to projective equivalence. There is an obvious compactification of $\Xi$ to $\mathbb{P}^2$. As outlined by Deligne and Mostow in \cite{Deligne:1986aa} this compactification is obtained by adding semi-stable configurations, which are configurations in which points may come together in a prescribed way.

Observe that on the curve side smoothness corresponds to the non-vanishing of the discriminant, which prevents multiple roots in the describing polynomial. Introducing semi-stable configurations removes this condition and allows higher multiplicities of these roots. This leads to degenerate curves.

There are three types of degeneration of the curve $$C_{Shiga}: y^3 = x(x-\xi_{0})(x-\xi_{1})(x-\xi_{2}) .$$ We denote these by $C^{deg}$ with subscripts $(2,1,1), (2,2)$ or $(3,1)$ indicating the multiplicities of the roots. Note again that it is not allowed that all four roots fall together as $[0,0,0]\not\in \mathbb{P}^2$.
\bigskip

Let $R' = \mathbb{Z}[\zeta_3,\tfrac{1}{6}][ \xi_{0}, \xi_{1}, \xi_{2}]$. We look at the individual degenerations.

%
%
%
%
\subsubsection*{The (2,1,1)-case}
%
%
%
%
%

If two roots coincide, the curve acquires a double point, decreasing the genus. Without loss of generality we assume that $\xi_{1}=\xi_{2}$; then the differentials
\[
\frac{(x-\xi_{1})dx}{y^{2}},\quad\frac{dx}{y}
\] 
remain holomorphic, so the genus is two, hence the curve is hyperelliptic. To find an affine hyperelliptic model $C'_{hyp}$, we first move the double root to zero, so that we work with $$y^3 = x^2(x-\xi_{0}+\xi_{1})(x+\xi_{1})$$ instead (without introducing new coordinates here). Then, setting $y=tx$ this implies
\begin{align*}
0= x^4 + (2\xi_{1}-\xi_{0}-t^3)x^3 + (\xi_{1}^2-\xi_{0}\xi_{1})x^2
\end{align*}
where we can divide by $x^2$ to obtain
\begin{align}		\label{EqCompletingSquare}
0= x^2 + (2\xi_{1}-\xi_{0}-t^3)x + (\xi_{1}^2-\xi_{0}\xi_{1}).
\end{align}
Assuming characteristic $\neq 2$, we can complete the square by setting 
\begin{align*} 		
s=2x+(2\xi_{1}-\xi_{0}-t^3),
\end{align*} 
which, using \eqref{EqCompletingSquare}, gives the desired affine model 
\begin{align}
C'_{hyp} : \; s^2 = t^6 -2(2\xi_{1}-\xi_{0})t^3 +\xi_{0}^2
\end{align}
of the hyperelliptic curve, which is of the form \eqref{modular version of Bolza's family}. Conversely, we have
\[
x=\tfrac{s+t^{3}+\xi_{0}-2\xi_{1}}{2}\quad y=t\tfrac{s+t^{3}+\xi_{0}-2\xi_{1}}{2},
\]
and a straightforward calculation shows that the distinguished holomorphic differentials are related via
\[
\frac{dx(s,t)}{y(s,t)}=3\frac{tdt}{s}.
\]

\bigskip

Finally let us look at the image of the (smooth) infinite point of the degenerate Picard curve of genus two in the hyperelliptic description. For this we use the model $$C''_{hyp}: v^2 = 1-2(2\xi_{1}-\xi_{0})u^3 + \xi_{0}^2u^6,$$
where $u= t^{-1}$ and $v= s/t^3$. So we have 
\begin{align}	
u&=\frac{x}{y} \label{EqU} \\ 
v&= \frac{2x +2\xi_{1}-\xi_{0}-t^3}{t^3} = \frac{2x+2\xi_{1}-\xi_{0}}{t^3} -1. \label{EqV} 
\end{align}
To determine $u$ we look at the cube of \eqref{EqU}
\begin{align*}
u^3&=\frac{x^3}{y^3} = \frac{x^3}{x^4+(2\xi_{1}-\xi_{0})x^3 +(\xi_{1}^2-\xi_{0})x^2}
\end{align*}
and use l'Hospital's rule
\begin{align*}
u^3&=\lim_{x\to \infty} \frac{3\cdot 2\cdot 1}{24x +6(2\xi_{1}-\xi_{0})} =0
\end{align*}
to deduce that $u=0$. To determine $v$ we use \eqref{EqV} and l'Hospital's rule 
\begin{align*}
v&=\lim_{x\to \infty}\frac{(2x\cdot x^3) + (2\xi_{1}-\xi_{0})\cdot x^3}{x^4+(2\xi_{1}-\xi_{0})x^3 +(\xi_{1}^2-\xi_{0})x^2} -1\\
&= 2+0 -1 = 1.
\end{align*}
Hence the point at infinity maps to $(u,v)=(0,1)\in C''_{hyp}$, which corresponds to the point at infinity in the affine model $C'_{hyp}$. 

\begin{rmk}
A similar computation shows that the singular point \mbox{$x=\xi_{1}$} on the Picard curve maps to  $(u,v)=(0,-1)\in C''_{hyp} $.  Note that $x=\xi_{1}$ means that $x\to 0$ in the above computation, since we changed coordinates $``x=x+\xi_{1}"$ without introducing new notation.
\end{rmk}

%
%
%
%
\subsubsection*{The (2,2)-case}
%
%
%
%

The degeneration $C^{deg}_{(2,2)}: y^3 = x^2(x-\xi_{0})^2$ of $C$ has genus one, is thus an elliptic curve. 

Let $R'$ be as above, then $$\mathcal{O}_{C^{deg}_{(2,2)}} = R'[x,y]/(y^3 - x^2(x-\xi_{0})^2)$$ is the coordinate ring of $C^{deg}_{(2,2)}$.
The normalization  $C_{(2,2)}^{norm}$ of $C^{deg}_{(2,2)}$ has coordinate ring $$\mathcal{O}_{C_{(2,2)}^{norm}} = \raisebox{1ex}{\textit{R'}[\textit{x,\,y,\,t}]\big/}\!\! \left( y^3 - x^2(x-\xi_{0})^2, ty=x(x-\xi_{0})\right)$$ exhibiting the Weierstrass form of the associated elliptic curve to be $$C_{(2,2)}^{norm} : Y^2 -\xi_{0}Y = X^3$$ for $Y=x$ and $X=t$.

%
%
%
%
\subsubsection*{The (3,1)-case}
%
%
%
%

By the same means we check that for the degeneration $$C^{deg}_{(3,1)}: y^3 = x(x-\xi_{0})^3$$ the normalization has coordinate ring
\begin{align*}\mathcal{O}_{C_{(3,1)}^{norm}} &= \raisebox{1ex}{\textit{R'}[\textit{x,\,y,\,t,\,u}]\big/}\!\! \left( y^3 - x^2(x-\xi_{0})^2, ty=x(x-\xi_{0}), u^2 = t, u^3 = x \right) \\
&\cong R'[u].
\end{align*}
In the first step we partially resolve the singularity so that the curve becomes the cuspidal elliptic curve $$ t^3=x^2,$$which we then resolve in the standard way.

%
%
%
%
\subsubsection*{Summary and Consequences}
%
%
%
%

As mentioned above, the curves parametrized by the complement of the smooth locus are degenerate. Turning to the interpretation of $\mathbb{C}H^2/ \Gamma'$ classifying structured abelian three-folds, we find that that points which are not in the image of the smooth locus (under Picard's period map) correspond to products of compatibly structured abelian varieties of lower dimension and cusps. Besides cusps, there are two possible product cases:
\begin{itemize}
\item[(a)] a product $A_2 \times E$, where $A_2$ is a structured abelian two fold and $E$ is a structured elliptic curve, i.e.\  an elliptic curve with complex multiplication by $\mathbb{Z}[\zeta_3]$.
\item[(b)] a triple product $E^{\times 3}$ of structured elliptic curves.
\end{itemize}

The degenerations of type $(2,1,1)$ correspond to products $A_2 \times E$, degenerations of type $(2,2)$ correspond to triple products $E^{\times 3}$, and degenerations of type $(3,1)$ correspond to cusps. 

There are precisely $\binom{4}{2}=6$ possibilities for common roots of type $(2,1,1)$, i.e.\ there are also $6$ hyperplanes in the ball quotient where varieties are of type\! $(a)$. Further, there are $\tfrac{1}{2}\cdot \binom{4}{2}= 3$ elliptic points, i.e.\ products of type\! $(b)$, and $\binom{4}{3}=4$ cusps. The overall arrangement is nicely pictured in \cite{Shiga:1988aa}:

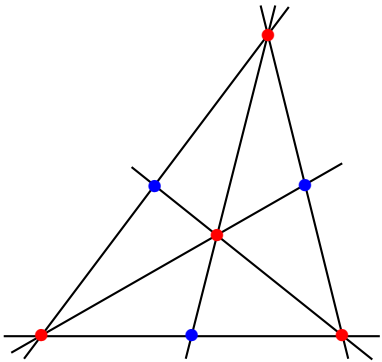
\begin{center}
\begin{tikzpicture}[scale=1.0]


\draw[thick] (-2.5, 1) -- (2.5,1);
\draw[thick] (-2.23, 0.7) -- (1.29,5.38);
\draw[thick] (2.08, 0.7) -- (0.915,5.4);

\draw[thick] (-0.08, 0.7) -- (1.11, 5.4);
\draw[thick] (-0.8, 3.25) -- (2.4,0.69);
\draw[thick] (-2.4, 0.78) -- (2,3.3);

\node [red] at (-2, 1) {$\bullet$};
\node [red] at (1.016, 5) {$\bullet$};
\node [red] at (2, 1) {$\bullet$};
\node [red] at (0.338, 2.335) {$\bullet$};

\node [blue] at (0,1) {$\bullet$};
\node [blue] at (1.508,3) {$\bullet$};
\node [blue] at (-0.491,2.99) {$\bullet$};

\end{tikzpicture}
\end{center}

Here the lines are the hyperplanes, the intersection of two hyperplanes, i.e.\ blue dots, are the elliptic points and the triple intersection points, i.e.\ the red dots, are the cusps.

%
%
%
%
%
%
%
%
%
\subsubsection*{Holzapfel's Family }
%
%
%
%
%
There are isomorphic curves among the Shiga family (analytically para-metrized by $\CC P^{2}\setminus\{4\ pts.\}$). According to Shiga \cite[p.\ 331]{Shiga:1988aa}, there are isomorphisms  \begin{align*}
PU(2,1;\ZZ[\zeta_{3}])/PU(2,1;\ZZ[\zeta_{3}])[\sqrt{-3}]&\cong S_{4}, \\ SU(2,1;\ZZ[\zeta_{3}])/SU(2,1;\ZZ[\zeta_{3}])[\sqrt{-3}]&\cong S_{4}.
\end{align*}
The latter implies $$S_4 \cong SU(H_3, \mathbb{F}_3) \cong SO(H_3, \mathbb{F}_3),$$ so that then we also find the reduction map $$ U(H_3, \mathbb{Z}[\zeta_3]) \to U(H_3, \mathbb{F}_3) \cong O(H_3, \mathbb{F}_3) \cong C_2\times SO(H_3, \mathbb{F}_3) \cong C_2\times S_4$$ to be surjective, leading to the short exact sequence
$$ U(H_3, \mathbb{Z}[\zeta_3])[\sqrt{-3}] \to U(H_3, \mathbb{Z}[\zeta_3]) \to C_2 \times S_4.$$

The action of $S_{4}$ induces a permutation action on the four cusps of $U(H_{3};\ZZ[\zeta_{3}])[\sqrt{-3}]$. To take care of the $S_{4}$ action, note that by replacing $x$ by $x+\tfrac{1}{4}(\xi_{0}+\xi_{1}+\xi_{2})$ in \eqref{Shiga family}, the sum of the four roots becomes zero; expanding, we obtain the family 
\begin{equation}\label{full Picard family}
C_{Pic}\: : \; y^3 = x^4 + G_2x^2 + G_3x + G_4,
\end{equation}
where $G_{i}$ is $(-1)^{i}$ times the $i^{th}$ elementary symmetric polynomial in the shifted roots. Thus, the coefficients of the curve equation lie in an $S_{4}$--invariant subring of automorphic forms for $U(2,1;\ZZ[\zeta_{3}])[\sqrt{-3}]$, which in turn can be identified with a ring of automorphic forms for the full group $U(2,1;\ZZ[\zeta_{3}])$.

More precisely, the factor group $\Gamma/\Gamma'\cong C_{2}\times S_{4}$ acts on the ring
\[
M_{*}^{\CC}(\Gamma')\cong\CC[\phi_{0},\phi_{1},\phi_{2}].
\] 
Recall that there are two irreducible three-dimensional representations of $S_{4}$ (up to equivalence): one of them, say $\pi$, is obtained by restricting the natural permutation action on $\CC^{4}$ to the invariant subspace $\{x_{1}+x_{2}+x_{3}+x_{4}=0\}$, and the other one, say $\pi'$, differs by the one-dimensional sign representation.  Clearly, we have
\[
\CC[\phi_{0},\phi_{1},\phi_{2}]^{\pi(S_{4})}\cong\CC[G_{2},G_{3},G_{4}].
\]
Since $\diag(-1,-1,-1)\in\Gamma$ acts as multiplication by $-1$ on $M_{6}^{\CC}(\Gamma')$,  and since $\pi$ and $\pi'$ induce the same {\em projective} representation, we conclude
\[
M_{*}^{\CC}(\Gamma)\cong M_{*}^{\CC}(\Gamma')^{C_{2}\times S_{4}}\cong\CC[G_{2},G_{4},G_{3}^{2}].
\]
\begin{rmk} 	\label{RemarkCharacter}
Given a character $\chi\colon\Gamma\to\CC^{*}$, we could consider forms for $\Gamma$ satisfying a modified transformation rule by including a factor $\chi(\gamma)^{k}$ on the RHS of \eqref{transformation property}. As explained in \cite{Holzapfel:1995aa}, the ring $\CC[G_{2},G_{3},G_{4}]$ can be interpreted as the ring of automorphic forms with character
\[
\chi\colon\Gamma\to\{\pm1\},\ g\mapsto(\det g)^{3}\sign(g),
\]
where $\sign(g)$ is the sign representation applied to the image of $g$ in $S_{4}$.
\end{rmk}

\begin{rmk} 
Shiga determines the subring of forms for the full group $\Gamma$ using invariant theory (cf.\ \cite[Proposition II-5]{Shiga:1988aa}); unfortunately, he denotes  the invariants $G_{i}$ as well, leading to a conflict in notation, viz.\ the $G_{i}$'s we defined here differ from the ones given by Shiga as follows:
\[
G_{2}^{S}=-8G_{2},\ G_{3}^{S}=-8G_{3},\ ((G_{2}^{S})^{2}-4G_{4}^{S}=256G_{4}.
\]
\end{rmk}

Observe that smoothness of the family \eqref{full Picard family} is still ensured by the non-vanishing of the  discriminant, which admits the expression
$$ \Delta_{C_{Pic}}=16G_{2}^{4}G_{4}-4G_{2}^{3}G_{3}^{2}-128G_{2}^{2}G_{4}^{2}+144G_{2}G_{3}^{2}G_{4}-27G_{3}^{4}+256G_{4}^{3}.$$
Moreover, there is still the action of the cyclic group $C_{3}$ on the curve, generated by
\[
(x,y)\mapsto(x,\zeta_{3}y),
\]
and the induced action on the space of differentials singles out the space spanned by $dx/y$. As  before, the Jacobians of these curves are ppavs with  $\ZZ[\zeta_{3}]$--endomorphisms of type $(2,1)$. More precisely, we can compose the inverse of the map \eqref{PeriodIso} with the natural projection $$ \mathbb{P}^2 \to \mathbb{P}^2/S_4.$$ 
The map $$(\mathbb{C}H^2/\Gamma)^* \xrightarrow{\cong} \mathbb{P}^2/S_4,$$ induced by passing to the quotient is hence an isomorphism, too.

\subsection{Topological Automorphic Forms}
In this section we describe how to associate TAF-type cohomology theories to isomorphism classes of Picard curves. The coefficients of these theories are rings of automorphic forms for the full arithmetic unitary group. We construct a genus targeting these rings and check Landweber's criterion.

%
%
%
%
\subsubsection*{Local Parameter}				\label{SectionLocalParameter}
%
%
%
%
We want to choose a local parameter at the point at infinity on $C$, which is again a fix point of the $C_3$-action on $C$. As $C$ is a \textit{family} of curves, note that the point at infinity is smooth throughout the family, i.e.\ is smooth in all fibers. Since our family is not defined over a field, the localization at infinity may not be a discrete valuation ring, so it may still not be clear what a local parameter is. We offer the following definition, just forcing things to be as usual.

\begin{df}
Let $C$ be a smooth algebraic curve over ${\rm Spec}(R)$ for a noetherian integral domain $R$. Let further $P$ be a smooth point of the family, such that the localization $\mathcal{O}_{C, (P)}$ is a local ring with principal maximal ideal $\mathfrak{m}$, then a generator $u$ of $\mathfrak{m}$ is called a \textit{local parameter} on $C$ at $P$.
\end{df}
Note that the local ring $S=\mathcal{O}_{C}(C)_{(P)}$ is not necessarily a discrete valuation ring, though in most examples it will be. If not, then, as it is noetherian, there will be a natural number $n$ such that $\{ S, \mathfrak{m}, \ldots, \mathfrak{m}^n=0 \}$ are all ideals in $S$.

\begin{rmk}
The idea behind a local parameter at a point is to map a formal neighborhood of the origin in $\mathbb{A}^1$ onto a formal neighborhood of that point on the curve. It is therefore obvious that local parameters behave well under rational, even analytic, maps, since formal parameters are also generators of the maximal ideal in the formal completion of the localization with respect to the maximal ideal.
\end{rmk}

Thus, if we want to find a local parameter around the point at infinity on $C_{Pic}$, instead of doing this directly, we consider the curve $$C_{(u,v)}\; \colon\; v^3 = 1 + G_2u^6 + G_3u^9 + G_4u^{12}$$ given by setting $x=u^{-3}$ and $y=vu^{-4}$. A local inverse $C_{Pic} \to C_{(u,v)}$ is now an analytic map, as it requires a third root of an element which is a unit in the localization. Note that this root exists in the formal completion and that the point $P=(0,1)$ maps to the point at infinity under the inverse map. Note further, that the distinguished differential $dx/y$ maps to $du/v$, and that it is the unique holomorphic differential (up to scale) which does not vanish at infinity.
Again, this implies that the Torelli map identifies $ T_\infty C_{Pic} $ $C_3$-equivariantly with the $1$-dimensional summand ${\rm Lie}^-({\rm Jac}(C_{Pic}))$.

To find a local parameter we consider the ideal $I=(u, v-1)$, describing the point $P$, in the coordinate ring of $C':=C_{(u,v)}$. Let $$R=\mathbb{Z}[\tfrac{1}{6}, \zeta_3][G_2, G_3, G_4, \Delta_C^{-1}]$$ then
$$ S:=\mathcal{O}_{C'}(C') = \raisebox{1ex}{\textit{R}[\textit{u,v}]\big/}\!\! \left( v^3-1-G_2u^6 -G_3u^9 - G_4u^{12})\right). $$
and 
\begin{align}
v^3 &= 1+G_2u^6 +G_3u^9 + G_4u^{12}\nonumber \\ 
\Leftrightarrow \;(v-1)(v^2+v+1) &= u^3\cdot p(u^3)  \label{LocParU(2,1)}
\end{align}
where $p\in \mathbb{Z}[x]$ and $(v^2+v+1)\not\in I = \mathfrak{m}_{S_{(I)}} \subset S_{(I)}$ is a unit in the localization. Hence, we find  $$ (v-1) \subset (u^3)\subset (u),$$ so $I=(u)$ in the localization $S_{(I)}$ and $u$ is a local parameter at $P$.

\begin{rmk}\label{standard parameter for Picard}
Note that there is another way of getting the local parameter: In the affine chart around $[0,1,0]$, the homogeneous equation $ZY^{3}=X^{4}+G_{2}X^{2}Z^{2}+G_{3}XZ^{3}+G_{4}Z^{4}$ becomes 
\[
\tilde{z}=\tilde{x}^{4}+G_{2}\tilde{x}^{2}\tilde{z}^{2}+G_{3}\tilde{x}\tilde{z}^{3}+G_{4}\tilde{z}^{4},
\]
where $\tilde{x}=X/Y=x/y$ and $\tilde{z}=Z/Y=1/y$. Clearly, $\tilde{x}$ is a local parameter. On the other hand, let
\[
v=(1+G_{2}u^{6}+G_{3}u^{9}+G_{4}u^{12})^{1/3} \in \mathbb{Z}_{(p)}[G_{2},G_{3},G_{4}][\![u]\!],
\]
and put $\tilde{x}=x/y=u^{-3}/(vu^{-4})=u/v$, $\tilde{z}=1/y=u^{4}/v$. Since $v$ is a unit in $\mathbb{Z}_{(p)}[G_{2},G_{3},G_{4}][\![u]\!]$, this shows that $u$ is a local parameter as well.
\end{rmk}

%
%
%
%
\subsubsection*{The $p$-complete Point of View}
%
%
%
%
By the same arguments as in Section \ref{SecHyperPcomplete} we find the following.
\begin{prop}

For (rational) primes which split $p=u\bar{u}$ in the imaginary quadratic number field $\mathbb{K}= \mathbb{Z}[\zeta_3]\otimes \QQ$, the formal group $\mathbb{G}_{{\rm Jac}(C)}$ splits $$\mathbb{G}_{{\rm Jac}(C)} \cong \mathbb{G}^{+} \times \mathbb{G}^{-}.$$ The dimension of the formal group $\mathbb{G}^{-}$ is one.

\end{prop}

\begin{proof}
The proof is exactly the same as for Corollary \ref{CorPcompleteSplittingTheFormalGroup}, but note that now ${\rm dim}(\mathbb{G}_{{\rm Jac}(C)})$ is $3$, hence the $\zeta_3$-action actually distinguishes only one (out of three) differential(s).
\end{proof}

Via the theory of \cite{Behrens:2010aa} this  $1$-dimensional summand gives rise to a $p$-complete complex oriented cohomology theory.

%
%
%
%
\subsubsection*{Formal Logarithm}
%
%
%
%

We define a rational genus 
$$ \varphi^{P} \colon MU_*^{\QQ} \to \QQ[G_{2},G_{3},G_{4}]$$
using the distinguished differential $du/v$ as derivative of the logarithm,
\begin{align*}
\log'_{\varphi^{P}}(u) = (1 + G_2u^6 + G_3u^9 + G_4u^{12})^{-1/3}.
\end{align*}
As in the warm up in Section \ref{SectionU11FormalLog}, we have an integrality statement.
\begin{thm}
For each prime $p\equiv1\!\mod3$, the image of $BP_*$ under $\varphi^{P}$ is contained in the subring of $p$-local automorphic forms:
$$ \varphi^{P}(BP_*) \subset M^{\mathbb{Z}_{(p)}}_*(\Gamma)  \cong\mathbb{Z}_{(p)}[G_2, G_4, G_3^{2}]\subset \mathbb{Z}_{(p)}[G_2, G_{3}, G_4].$$

\end{thm}
\begin{proof}
Having a rational genus and a formal group over $\mathbb{Z}[G_2, G_3, G_4]^\wedge_p$, we prove the claim 
by the exact same argument as presented in the proof of Theorem \ref{ThmHyperIntegrality}.
\end{proof}

\begin{rmk}\label{dimensional reasons}
Since
\[
\varphi^{P}(BP_*) \subset \mathbb{Z}_{(p)}[G_2, G_4, G_3^{2}],
\]
we may pass to the $p$--typicalization to obtain a group law defined over this smaller ring.
\end{rmk}

%
%
%
%
\subsubsection*{p-local TAF}
%
%
%
%
Recall that smoothness of the Picard curves \eqref{full Picard family} requires that the discriminant
\[
\Delta_{C}=16G_{2}^{4}G_{4}-4G_{2}^{3}G_{3}^{2}-128G_{2}^{2}G_{4}^{2}+144G_{2}G_{3}^{2}G_{4}-27G_{3}^{4}+256G_{4}^{3}
\]
is nonzero. Write $M_{*}^{\ZZ_{(p)}}(\Gamma, \chi)\cong\ZZ_{(p)}[G_{2},G_{3},G_{4}]$ for the polynomial ring of automorphic forms with character $\chi$ of Remark \ref{RemarkCharacter}.

\begin{cor}		\label{CorTAF at p=7}
Let $p=7$. Then $\varphi^{P}$ gives $M_{*}^{\ZZ_{(7)}}(\Gamma, \chi)[\Delta_{C}^{-1}]$ the structure of a Landweber exact $BP_*$-algebra of height three. In particular, the functors
\[
TAF^{U(2,1;\ZZ[\zeta_3])}_{(7),*}(\cdot):=BP_{*}(\cdot)\otimes_{\varphi^{P}}M^{\ZZ_{(7)}}_{*}(\Gamma, \chi)[\Delta_{C}^{-1}]
\]
define a homology theory.
\end{cor}

\begin{proof} We check Landweber's criterion: Multiplication by 7 is injective. Next, observe that  $M^{\ZZ_{(7)}}_{*}(\Gamma, \chi)/(7)\cong\mathbb{F}_{7}[G_{2},G_{3},G_{4}]$ is an integral domain, so multiplication by $v_{1}=-\tfrac{1}{3}G_{2}\equiv2G_{2}\mod(7)$ is injective. Similarly, 
\[
v_{2}\equiv G_{4}^{4}-2G_{3}^{4}G_{4}\mod(7,v_{1})
\]
is nonzero, hence regular on the integral domain $M^{\ZZ_{(7)}}_{*}(\Gamma, \chi)/(7,v_{1})\cong\mathbb{F}_{7}[G_{3},G_{4}]$. To establish the claim, it therefore remains to check that $v_{3}$ is invertible mod $(7,v_{1},v_{2})$: we have $\Delta_{C}\equiv G_{3}^{4}+4G_{4}^{3}$, and a computation shows
\[
v_{3}^{2}\equiv (G_3^{34}(6G_3^4+2G_4^3))^{2}\equiv G_{3}^{76}+4G_3^{72}G_{4}^{3}\mod(7,v_{1},v_{2}),
\]
which coincides with $\Delta_{C}^{19}\mod(7,v_{1},v_{2})$.
\end{proof}

\begin{cor}
Let $p=13$. Then $\varphi^{P}$ gives $M_{*}^{\ZZ_{(13)}}(\Gamma, \chi)[\Delta_C^{-1}]$ the structure of a Landweber exact $BP_*$-algebra of height three. In particular, the functors
\[
TAF^{U(2,1;\ZZ[\zeta_3])}_{(13),*}(\cdot):=BP_{*}(\cdot)\otimes_{\varphi^{P}}M^{\ZZ_{(13)}}_{*}(\Gamma, \chi)[\Delta_{C}^{-1}]
\]
define a homology theory.

\end{cor}

\begin{proof}
We proceed as in the proof of Corollary \ref{CorTAF at p=7}. Multiplication with $13$ is injective on $M^{\mathbb{Z}_{(13)}}_*(\Gamma, \chi)$ and $M^{\mathbb{Z}_{(13)}}_*(\Gamma, \chi)/(13)$ is an integral domain, so multiplication with 
\begin{align*}
v_1&\equiv 6\,{G_{2}}^{2}+4\,G_{4} \mod 13
\end{align*}
is injective on it. Further, $M^{\mathbb{Z}_{(13)}}_*(\Gamma, \chi)/(13, v_1) \cong \mathbb{F}_{13}[G_2, G_3]$ is also an integral domain and thus again multiplication with 
\begin{align*}
v_2 &\equiv 12\,{G_{2}}^{28}+4\,{G_{2}}^{25}{G_{3}}^{2}+10\,{G_{2}}^{22}{G_{3}}^{4}+4\,{G_{2}}^{19}{G_{3}}^{6}
+\\&+7\,{G_{2}}^{16}{G_{3}}^{8}+10\,{G_{2}}^{13}{G_{3}}^{10}+4\,{G_{2}}^{10}{G_{3}}^{12}+\\&+6\,{G_{2}}^{7
}{G_{3}}^{14}+3\,{G_{2}}^{4}{G_{3}}^{16}+8\,G_{2}{G_{3}}^{18} \mod (13, v_1)
\end{align*} 
is injective. Since 
\begin{align*}
v_3&\equiv 2\,{G_{2}}^{27}{G_{3}}^{226}+6\,{G_{2}}^{24}{G_{3}}^{228}+11\,{G_{2}}^{21}{G_{3}}^{230}+3\,{G_{2}}
^{18}{G_{3}}^{232}+\\&+9\,{G_{2}}^{15}{G_{3}}^{234}+2\,{G_{2}}^{12}{G_{3}}^{236}+3\,{G_{2}}^{9}{G_{3}}
^{238}+9\,{G_{2}}^{6}{G_{3}}^{240}+\\&+9\,{G_{2}}^{3}{G_{3}}^{242}+8\,{G_{3}}^{244} \mod(13,v_1,v_2)
\end{align*}
one finds that $\Delta_C^{366} \equiv -(v_3^6) \mod(13,v_1,v_2)$, which proves the claim.
\end{proof}

%
%
%
\subsection{Reduction of Height}\label{restriction and degeneration}
%
%
%
%

Recall from Section \ref{SectionSignature n-1,1} that the natural inclusion $U(H_{2})\to U(H_{3})$ induces an embedding $\mathbb{L}_{1}\to\mathbb{L}_{2}$. In particular, we may pull back the automorphic forms on $\mathbb{L}_{2}$ along this inclusion; comparing $q$--expansions, we conclude
\[
\phi_{0}\left(\begin{smallmatrix}z_{0}\\0\end{smallmatrix}\right)=c\cdot E_{1}^{3}(z_{0}/\sqrt{-3}),\ \phi_{1}\left(\begin{smallmatrix}z_{0}\\0\end{smallmatrix}\right)=\phi_{2}\left(\begin{smallmatrix}z_{0}\\0\end{smallmatrix}\right)=c\cdot E_{3}(z_{0}/\sqrt{-3}),
\]
where $c=\theta^{3}[\begin{smallmatrix}1/6\\1/6\end{smallmatrix}](-\zeta_{3}^{2})$.
Absorbing this constant $c$ and incorporating the identification $\mathbb{H}_{1}\cong\mathbb{L}_{1}$ therefore yields a ring homorphism
\[
r\colon\ZZ_{(p)}[G_{2},G_{3},G_{4}]\to\ZZ_{(p)}[\kappa,\lambda],
\]
which takes the form
\begin{align*}
r(G_{2})&=-\tfrac{3}{8}\lambda+\tfrac{27}{2}\Delta_{6} \\
r(G_{3})&=+\tfrac{1}{8}\kappa\lambda \\
r(G_{4})&=-\tfrac{1}{256}\kappa^{2}(3\lambda+108\Delta_{6}).
\end{align*} 
This allows us to compare the `Picard genus' $\varphi^{P}$ constructed above to the genus $\varphi^{L}$ constructed before:
\begin{thm}
For each prime $p\equiv1\!\mod3$, the genus $r\circ\varphi^{P}$ is equivalent to $\varphi^{L}$; more precisely, the associated group laws are isomorphic over $\ZZ_{(p)}[\kappa,\lambda]$.
\end{thm}
\begin{proof} Applying $r$ to the coefficients of the Picard curve \eqref{full Picard family}, we obtain the family
\[
y^{3}=((x-c)^{2}-(a-c+b-c)(x-c)+(a-c)(b-c))(x-c)^{2},
\]
where $a=-\tfrac{1}{4}(E_{1}^{3}+2E_{3})$, $b=-\tfrac{1}{4}(2E_{3}-3E_{1}^{3})$, $c=-\tfrac{1}{4}(E_{1}^{3}-2E_{3})$. As observed before, this curve admits a hyperelleptic model,
\[
s^{2}=E_{1}^{6}+2(E_{1}^{3}-2E_{3})t^{3}+t^{6},
\]
where $t=y/(x-c)$. We also know that $t^{-1}=(x-c)/y$ serves as a local parameter around a point at infinity of this hyperelliptic curve. On the other hand, it also works as a local parameter for the Picard curve: in the notation of Remark \ref{standard parameter for Picard}, we have $t^{-1}=\tilde{x}-c\tilde{z}$, and $\tilde{z}=O(\tilde{x}^{4})$. Thus, expressing this local parameter in terms of the parameter $u$ induces the required isomorphism of formal group laws.
\end{proof}

It follows that (at least at small primes) $r\circ\varphi^{P}$ gives $\ZZ_{(p)}[\kappa^{\pm1},\lambda^{\pm1}]$ the structure of a Landweber exact algebra of height two.

\begin{thm} Let $p=7,13$. Then the homomorphism $r$ induces a homomorphism of Landweber exact algebras
\[
\ZZ_{(p)}[G_{2},G_{3}^{\pm1},G_{4}]\to\ZZ_{(p)}[\kappa^{\pm1},\lambda^{\pm1}],
\]
decreasing the height by one.
\end{thm}
\begin{proof} First we note that it suffices to invert $G_{3}$ to obtain a Landweber exact algebra:
For $p=7$, it is easy to verify that
\[
\Delta_{C}^{57}\equiv G_{3}^{228}\mod(7,v_{1},v_{2}).
\]
Similarly, for $p=13$, we have $v_{3}^{18}\equiv-G_{3}^{4392}\mod(13,v_1,v_2)$.
We remark that $G_{3}$ decomposes as a product in the larger ring of automorphic forms with level structure,
\[
G_{3}=\tfrac{1}{8}(-\xi_{0}+\xi_{1}+\xi_{2})(\xi_{0}-\xi_{1}+\xi_{2})(\xi_{0}+\xi_{1}-\xi_{2}),
\]
i.e.\ the preimage of the zero locus of $G_{3}$ is a union of three hyperplanes in $\mathbb{P}^{2}$.
Since $r(G_{3})=\tfrac{1}{8}\kappa\lambda$, the claim now follows.
\end{proof}

\begin{rmk} Inverting $\kappa\lambda$ amounts to removing two points from the (compactified) moduli space. However, the point corresponding to the cusp is still there, so we can decrease the height even further by restricting to the degenerate curve $y^{2}=(1-x^{3})^{2}$.
\end{rmk}

\begin{rmk} Over the complex numbers, every automorphic form for $\Gamma$ admits a Taylor expansion around $z_{1}=0$. Our restriction homomorphism yields the zeroth term of this expansion, and it is easy to show that all higher coefficients are also complex automorphic forms for the group $G$ (of shifted weight).
\end{rmk}

\begin{rmk}\label{a supersingular Picard curve}
Consider a curve given by
\[
y^{3}=x(x-1)(x-\zeta_{3})(x-\zeta_{3}^{2})=x(x^{3}-1);
\]
if $p\equiv1\mod 3$ but $9\nmid (p-1)$, then the associated one-dimensional formal group law over $\mathbb{F}_{p}$ has height at least three, since $v_{1}=0=v_{2}$ for dimensional reasons; if the $p$--valuation of the coefficient of $u^{p^{3}-1}$ in $(1-u^{9})^{-1/3}$ happens to be 2, then the height is indeed three.
\end{rmk}

\appendix
\section{TAF with Level Structure}
Of course, it is possible to construct a $TAF$--type homology theory directly from the family of curves \eqref{Shiga family}; let us sketch the argument: In the usual affine chart around $[0,1,0]$, the homogeneous equation $ZY^{3}=X(X-\xi_{0}Z)(X-\xi_{1}Z)(X-\xi_{2}Z)$ becomes 
\[
\tilde{z}=\tilde{x}(\tilde{x}-\xi_{0}\tilde{z})(\tilde{x}-\xi_{1}\tilde{z})(\tilde{x}-\xi_{2}\tilde{z}),
\]
where $\tilde{x}=X/Y=x/y$ and $\tilde{z}=Z/Y=1/y$. Clearly, $\tilde{x}$ is a local parameter. On the other hand, let
\[
v=((1-\xi_{0}u^{3})(1-\xi_{1}u^{3})(1-\xi_{2}u^{3}))^{1/3}\in\ZZ_{(p)}[\xi_{0},\xi_{1},\xi_{2}][\![u]\!],
\]
and put $\tilde{x}=x/y=u^{-3}/(vu^{-4})=u/v$, $\tilde{z}=1/y=u^{4}/v$. Since $v$ is a unit in $\ZZ_{(p)}[\xi_{0},\xi_{1},\xi_{2}][\![u]\!]$, this shows that we may as well parametrize the curve near the infinite point by $u$. With this choice of parameter, the distinguished differential becomes

$$ \frac{du}{v} = \left( (1-\xi_{0}u^3)(1-\xi_{1}u^3)(1-\xi_{2}u^3) \right)^{-1/3}du.$$
As before, we define a rational genus $$\varphi^{S}\colon MU_*^{\QQ} \rightarrow M_*^{\QQ}(\Gamma')\cong\QQ[\xi_{0},\xi_{1},\xi_{2}] $$ by $$ {\rm log}'_{\varphi^{S}}(u) = \left( (1-\xi_{0}u^3)(1-\xi_{1}u^3)(1-\xi_{2}u^3) \right)^{-1/3}.$$

We have the following integrality statement:
\begin{thm}		\label{ThmIntGenusU(2,1)Level}
For each prime $p\equiv1\!\mod3$, the image of $BP_*$ under $\varphi^{S}$ is contained in the subring of $S_{3}$--invariant $p$-local automorphic forms:
$$ \varphi^{S}(BP_*) \subset M^{\mathbb{Z}_{(p)}}_*(\Gamma[\sqrt{-3}])^{S_{3}}  \cong\mathbb{Z}_{(p)}[\sigma_1, \sigma_2, \sigma_3], $$
where $\sigma_i:=\sigma_i(\xi_{0}, \xi_{1}, \xi_{2})$ is the $i$-th symmetric function in the roots $\xi_{0}, \xi_{1}, \xi_{2}$ of the polynomial describing the curve $C$.
\end{thm}

\begin{proof}
Having a rational genus and a formal group over $\mathbb{Z}[\sigma_1, \sigma_2, \sigma_3, \Delta_{C}^{-1}]^\wedge_p$, we prove the claim 
by the exact same argument as presented in the proof of Theorem \ref{ThmHyperIntegrality}
\end{proof}

\begin{rmk} It is possible to characterize an arithmetic subgroup $\Gamma''\subset\Gamma=U(H_{3};\ZZ[\zeta_{3}])$ with ring of automorphic forms isomorphic to $\CC[\sigma_{1},\sigma_{2},\sigma_{3}]$. To this end,  we remark that the group 
\[
\Gamma_{1}[\sqrt{-3}]=\{\gamma\in\Gamma: \gamma\equiv\left(\begin{smallmatrix}1&*&*\\&1&*\\&&1\end{smallmatrix}\right)\mod\sqrt{-3}\}
\]
contains the principal congruence subgroup of level $\sqrt{-3}$ as a normal subgroup. It is easy to verify that $\Gamma_{1}[\sqrt{-3}]/\Gamma[\sqrt{-3}]\cong C_{3}$, where the cyclic group is generated by the reduction of $\left(\begin{smallmatrix}1&-1&\zeta_{3}\\&1&1\\&&1\end{smallmatrix}\right)\in\Gamma_{1}[\sqrt{-3}]$. This cyclic group induces a cyclic permutation of the forms $\phi_{0},\phi_{1},\phi_{2}$ and can be enlarged to $S_{3}\cong C_{3}\rtimes C_{2}$, e.g.\ by taking the $C_{2}$ generated by $\diag(1,-1,1)$. Since the action of the latter element indeed permutes $\phi_{1}$ and $\phi_{2}$, it follows that the desired arithmetic group is
\[
\Gamma''=\{\gamma\in\Gamma: \gamma\equiv\left(\begin{smallmatrix}1&*&*\\&*&*\\&&1\end{smallmatrix}\right)\mod\sqrt{-3}\}.
\]
\end{rmk}

Observe that in terms of the elementary symmetric polynomials, the discriminant reads
\[
\Delta_{C}=\sigma_{3}^{2}\left(\sigma_2^2\sigma_1^2-4\sigma_2^3-4\sigma_3\sigma_1^3+18\sigma_3\sigma_2\sigma_1-27\sigma_3^2\right).
\]
As before, we can construct a homology theory.
\begin{cor}\label{CorTAF at p=7 withLevel}
Let $p=7$ and let $M_{*}^{\ZZ_{(7)}}\!\cong\ZZ_{(7)}[\sigma_{1},\sigma_{2},\sigma_{3}]$ be the $S_{3}$--invariant subring of $7$--local automorphic forms for the congruence subgroup $U(H_{3};\ZZ[\zeta_{3}])[\sqrt{-3}]$. Then $\varphi^{S}$ gives $M_{*}^{\ZZ_{(7)}}[\Delta_{C}^{-1}]$ the structure of a Landweber exact $BP_*$-algebra of height three. In particular, the functors
\[
TAF^{U(2,1;\ZZ[\zeta_3])[\sqrt{-3}]}_{(7),*}(\cdot):=BP_{*}(\cdot)\otimes_{\varphi^{S}}M^{\ZZ_{(7)}}_{*}[\Delta_{C}^{-1}]
\]
define a homology theory.
\end{cor}
\begin{proof} Clearly, multiplication with 7 is injective on $M_{*}^{\ZZ_{(7)}}$. Moreover, $M_{*}^{\ZZ_{(7)}}/(7)\cong\mathbb{F}_{7}[\sigma_{1},\sigma_{2},\sigma_{3}]$ is an integral domain, hence multiplication by
\[
v_{1}= \frac{2}{9}\sigma_1^2 - \frac{1}{3}\sigma_2\equiv \sigma_{1}^{2}+2\sigma_{2}\mod(7)
\]
is injective. The latter relation allows us to identify $M_{*}^{\ZZ_{(7)}}/(7,v_{1})\cong\mathbb{F}_{7}[\sigma_{1},\sigma_{3}]$, which is therefore also an integral domain, and one verifies
\[
v_{2}\equiv 2\sigma_{1}^{16} + 2\sigma_{1}^{13}\sigma_{3} +4\sigma_{1}^{10}\sigma_{3}^{2}+6\sigma_{1}^{7}\sigma_{3}^{3} + 3\sigma_{1}^{4}\sigma_{3}^{4}+4\sigma_{1}\sigma_{3}^{5}\mod(7,v_{1}).
\]
Furthermore, a computation shows
\[
v_{3}\equiv 2\sigma_{1}^{12}\sigma_{3}^{34} + 4\sigma_{1}^{9}\sigma_{3}^{35} +\sigma_{1}^{6}\sigma_{3}^{36}+6\sigma_{3}^{38}\mod(7,v_{1},v_{2}),
\]
\[
v_{3}^{2}\equiv 4\sigma_{1}^{12}\sigma_{3}^{72}+\sigma_{1}^{9}\sigma_{3}^{73}+2\sigma_{1}^{6}\sigma_{3}^{74}+\sigma_{3}^{76}\mod(7,v_{1},v_{2})
\]
and also
\[
\Delta_{C}^{19}\equiv 4\sigma_{1}^{12}\sigma_{3}^{72}+\sigma_{1}^{9}\sigma_{3}^{73}+2\sigma_{1}^{6}\sigma_{3}^{74}+\sigma_{3}^{76}\mod(7,v_{1},v_{2}).
\]
Thus, upon inverting $\Delta_{C}$, the reduction $v_{3}\mod(7,v_{1},v_{2})$ becomes a unit.
\end{proof}

\begin{rmk}
Note that $\Delta_{C}$ decomposes as a product, $\Delta_{C}=\sigma_{3}^{2}Q$. At least for $p=7$, it is not necessary to invert all factors; indeed, a computation shows
\[
v_{3}^{3}\equiv-\sigma_{3}^{114} \mod (7,v_{1},v_{2}),
\]
hence our genus turns $\ZZ_{(7)}[\sigma_{1},\sigma_{2},\sigma_{3}^{\pm1}]$ into a Landweber exact algebra of height three (since $v_{3}^{6}\equiv\Delta_{C}^{57}$, we also have $v_{3}^{3}\equiv -Q^{57}$, so we could invert $Q$ instead). Note that the zero locus of $\sigma_{3}$ can be identified with a union of three hyperplanes in $\CC P^{2}$, and the union of these hyperplanes corresponds to those Shiga curves for which zero is a root of multiplicity larger than one.
\end{rmk}

\bibliography{refbib_edited}
\end{document}